\documentclass{article}
\usepackage[
journal=AIHPD, 
lang=british,   %% Change to `american' if you use American English.
]{ems-journal}

\theoremstyle{plain}
\newtheorem{Teo}{Theorem}[section]
\newtheorem{Lem}[Teo]{Lemma}
\newtheorem{Cor}[Teo]{Corollary}
\theoremstyle{definition}

\newtheorem{Definition}[Teo]{Definition}

\newcommand{\DistTo}{\xrightarrow{
   \,\smash{\raisebox{-0.65ex}{\ensuremath{\scriptstyle\sim}}}\,}}

\numberwithin{equation}{section}

\begin{document}

%------
% Insert the title of your paper and (if necessary)
% a short title for the running head.
%------
\title{Ternary relations and their polytopes}
\titlemark{Ternary relations and their polytopes}

%------

%%%% Pls fill in all fields for each author
%%%% Label the authors by their position in the authors' list using {}
%%%% If you published any math paper ever, you have an MR Author ID.
%  Please look it up in three easy (and free) steps:
% 1. copy the bibliographic data of any published paper (co-)authored by you in the search field at https://mathscinet.ams.org/mathscinet/freetools/mref
% 2. Hit your name in the search result
% 3. Find your MR Author ID in the first row, copy it in the \mrid{} field
%%%% If you have not created your ORCID yet, you may like to do it now, pls copy it in the field \orcid{}
%%%% Abbreviate first names for the running head

\emsauthor{1}{
	\givenname{Aleksei}
	\surname{Lavrov}
	\orcid{0000-0001-5608-3907}}{A.~Lavrov}

%%%% Please provide detailed address info for each author
%%%% Use the same numbering as for \emsauthor above
%%%% Please look up the ROR ID of your institute here: https://ror.org
\Emsaffil{1}{
	\department{Higher School of Mathematics}
	\organisation{MIPT}
	\rorid{00v0z9322}
	\address{9  Institutskiy per.}
	\zip{141701}
	\city{Dolgoprudny}
	\country{Russia}
	\affemail{lavrov.an@mipt.ru}}

%------
% Insert your abstract.
%------
\begin{abstract}
Graev\footnote{M. M. Graev, not to be confused with his father M. I. Graev} introduced 
the construction of a convex polytope associated with a symmetric ternary relation. 
He showed that the number of left-invariant Einstein metrics on a homogeneous space under 
some conditions is no more than the normalized volume of certain polytope of such form. 
It happens that the construction of a cosmological polytope introduced 
by Arkani-Hamed, Benincasa and Postnikov for computation of the wave function of 
the Universe is the special case of the Graev construction. The paper is devoted to 
unification of these two theories from combinatorial perspective.
\end{abstract}

\maketitle

%------
% INSERT THE BODY OF THE PAPER HERE (except
% acknowledgments, funding info and bibliography)
%------

\section{Introduction}

A \textit{(symmetric) ternary relation} is a pair $\mathcal{T} = (\Sigma, R)$ consisted of 
a finite set $\Sigma \simeq \{1, ..., n\}$ and a collection of unordered triples 
$R \subset \Sigma \times \Sigma \times \Sigma / S_{3}$. We will denote the unordered 
triple contained elements $i$, $j$, $k \in \Sigma$ by $[i, j, k]$. Two ternary relations 
$\mathcal{T}_{1}$ and $\mathcal{T}_{2}$ are said to be isomorphic if there is a bijection 
$\Sigma_{1} \longrightarrow \Sigma_{2}$ which sends any triple from $R_{1}$ to the triple 
in $R_{2}$. One can associate with any symmetric ternary relation a convex polytope in 
the following way.

\begin{Definition}\label{TernaryPolytope}
Suppose that we have a ternary relation $\mathcal{T}=(\Sigma, R)$. Consider the vector 
space $\mathbb{R}^{n}$ with the basis $\textbf{1}_{1}, ..., \textbf{1}_{n}$ enumerated by 
elements of the set $\Sigma$. The \textit{ternary polytope} $P(\mathcal{T}) \subset
\mathbb{R}^{n}$ associated with $\mathcal{T}$ is the convex polytope of the form:
\begin{equation*}
	P(\mathcal{T}) := 
	\text{Conv}\bigg( 
		\Big\{ 
			\textbf{1}_{i} + \textbf{1}_{j} - \textbf{1}_{k}, \ \
			\textbf{1}_{i} - \textbf{1}_{j} + \textbf{1}_{k}, \ \
			- \textbf{1}_{i} + \textbf{1}_{j} + \textbf{1}_{k} \ \ | \ \ 
			\text{ for each } [i, j, k] \in R 
		\Big\} 
	\bigg).
\end{equation*}
\end{Definition}

\noindent
Note that $P(\mathcal{T})$ is a lattice polytope which lies in the affine hyperplane 
$\Big\{\sum\limits_{i = 1}^{n} x^{i} = 1\Big\} \subset \mathbb{R}^{n}$ where 
$x^{1}, ..., x^{n}$ are coordinates of $\mathbb{R}^{n}$ in the basis 
$\textbf{1}_{1}, ..., \textbf{1}_{n}$. The ternary polytopes naturally appear in 
the following three topics: left-invariant Einstein metrics on homogeneous space, 
cosmological polytopes and finite metric spaces.

\subsection{Left-invariant Einstein metrics on homogeneous spaces}

The Riemannian metric $g$ on the manifold $M$ is called Einstein if it satisfies the
differential equation $\text{Ric}~g = \Lambda \cdot g$, $\Lambda \in \mathbb{R}$ (see
\cite{Besse}). For left-invariant metrics on homogeneous space $M = G/H$ of Lie group
$G$ the Einstein equation becomes the system of rational equations. The problem of 
description of real solutions of these rational equations is very difficult
and there are a lot of papers devoted to this problem (see the comprehensive overview in 
\cite{Ar}). The typical situation in this area is that the certain list of metrics on
given $M$ is known, however, there are no ways to construct new metrics, or to prove that
this list is complete even using computer. For example, the number of left-invariant 
Einstein metrics on Lie group $SU(2) \times SU(2)$ is still unknown (see \cite{BCHL}). 

Graev in his paper \cite{GrUMN} studied the particular case of homogeneous spaces $M=G/H$ 
for which $H$ is a compact subgroup of $G$ and the so-called isotropy representation 
$H$ \rotatebox[origin=c]{270}{$\circlearrowleft$} $T_{[eH]}(G/H)$ 
has no equivalent irreducible components. In this case the Einstein equation takes 
the form of the system of Laurent polynomial equations. Instead of considering real 
solutions he considered complex solutions of this system which correspond to 
complex-valued left-invariant Einstein metrics on $M$. If $G$ is a compact semisimple Lie
group, he showed using the Bernstein-Kushnirenko theorem (see \cite{Bern, K}) that 
the number $\mathcal{E}(M)$ of isolated complex-valued left-invariant Einstein metrics on
$M$ (up to homothety) is no more than the normalized volume $\nu(M)$ of the Newton 
polytope $\text{Newt}(s_{M})$ of the scalar curvature $s_{M}$ which also becomes Laurent 
polynomial in the considered case. If the defect $\delta(M) := \nu(M) - \mathcal{E}(M) 
\geq 0$ vanishes then all complex-valued left-invariant Einstein metrics are isolated, 
but for the case of the positive defect $\delta(M) > 0$ he proved that the "lost" metrics 
in the number of $\delta(M)$ can be found in the so-called Inonu-Wigner contractions 
$M_{\gamma}$ of $M$ associated with faces $\gamma \subset \text{Newt}(s_{M})$ of 
the polytope $\text{Newt}(s_{M})$ which are non compact homogenous spaces. Moreover, he 
noted that the construction of $\text{Newt}(s_{M})$ essentially depends only on 
the symmetric ternary relation $\mathcal{T}(M)$ describing commutation relationships 
between irreducible components of the isotropy representation 
$H$ \rotatebox[origin=c]{270}{$\circlearrowleft$} $T_{[eH]}(G/H)$.
Based on this observation he introduced the construction of polytope $P(\mathcal{T})$ by 
arbitrary symmetric ternary relation $\mathcal{T}$ such that $\text{Newt}(s_{M}) =
P\big( \mathcal{T}(M) \big)$. This construction is presented in 
Definition \ref{TernaryPolytope}.

In the case of flag manifolds $M = G/H$, i. e. $H$ is a centralizer of any torus in $G$,
the corresponding ternary relation $\mathcal{T}(M)$ can be directly constructed by
the so-called $T$-root system $\Omega$ of $M$. A $T$-root system is some generalization of 
the notion of root system (see \cite{AP}), but we can think about it just as a finite 
configuration of vectors in some vector space. The construction of the ternary relation 
associated with $\Omega$ is the following:

\begin{Definition}\label{GraevTernaryRelation}
Let $\Omega = \{ \pm \alpha_{i} \ | \ i = 1, ..., n \} \subset \mathbb{R}^{d}$ be a finite 
configuration of vectors in $d$-dimensional vector space such that 
$\Omega = (-1) \cdot \Omega$. 
A ternary relation $\mathcal{T}(\Omega)$ is a pair $(\Sigma_{\Omega}, R_{\Omega})$ 
such that $\Sigma_{\Omega} = \{ 1, ..., n \}$ is the set of indices of $\Omega$, while
$R_{\Omega}$ contains all triples $[i, j, k]$ for which the equality 
$\pm \alpha_{i} \pm \alpha_{j} \pm \alpha_{k} = 0$ holds for some distribution of signs.
\end{Definition}

\noindent 
Note that considering different enumerations of vectors of $\Omega$ we obtain isomorphic 
ternary relations so we do not distinguish them. Various combinatorial properties of 
ternary polytopes of the form\footnote{To be precise, Graev considered the polytopes of 
the form $(-1) \cdot P(\Omega)$.} $P(\Omega):=P\big( \mathcal{T}(\Omega) \big)$ were 
deeply studied by Graev in \cite{GrTMMO1}. For some low rank root systems these polytopes 
can be described explicitly (see Fig. \ref{Ternary_polytopes}): $P(A_{1}) = P(B_{1})$ is 
considered just as a point, $P(A_{2})$ is a 2-dimensional triangle, $P(B_{2})$ is 
a 3-dimensional triangular prism, $P(A_{3})$ is a tetrahedron, $P(A_{4})$ is 
a 5-dimensional polytope isometric to the product of simplices $\Delta_{2} \times 
\Delta_{3}$. For many $T$-root systems of higher rank Graev provided complete list of 
faces and computed the normalized volumes (namely, for $A_2$, $B_2$, $BC_{2,1}$, $BC_2$, 
$G_2$, $A_3$, $C_{3,1}$, $C_{3,2}$, $C_3$, $B_3$, $A_4$). However, there are still a lot 
of open questions about the polytopes $P(\Omega)$ even for the simplest root system 
$\Omega = A_{n}$, $n \geq 5$. For example, the formula for normalized volume of $P(A_{n})$ 
is unknown and the complete description of faces of $P(A_{n})$ is not done. We call 
ternary polytopes of the form $P(\Omega)$ by \textit{Graev polytopes}.

\begin{figure}[t]
\includegraphics[width=.6\textwidth]{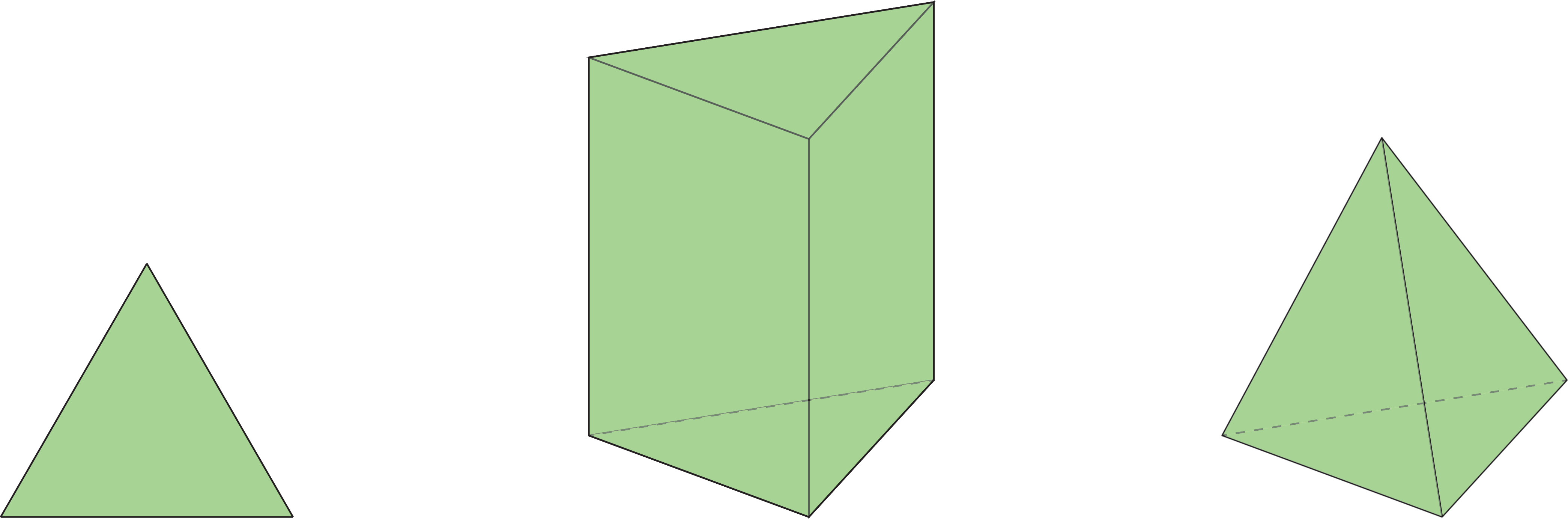}
\caption{Examples of ternary polytopes: $P(A_{2})$, $P(B_{2})$ and $P(A_3)$ (from left
to right).}\label{Ternary_polytopes}
\end{figure}

\subsection{Cosmological polytopes}

The study of another special class of ternary polytopes was initiated by Arkani-Hamed, 
Benincasa and Postnikov in \cite{A-HBP}. With any undirected graph $G$ considered as 
a special type of Feynman diagram they associated in some way a convex polytope $P(G)$ 
which was called \textit{a cosmological polytope}. They did not used explicitly the notion 
of ternary relation, but nevertheless their construction can be seen as the partial case 
of the construction of ternary polytopes. More precisely, it is possible to naturally 
construct the ternary relation $\mathcal{T}(G)$ by any graph $G$ such that 
$P(G) = P\big( \mathcal{T}(G) \big)$ as follows: 

\begin{Definition}\label{GraevTernaryRelation}
Let $G$ be an undirected graph with the set of vertices $V(G)$ and the set of edges 
$E(G)$. A ternary relation $\mathcal{T}(G)$ is a pair $(\Sigma_{G}, R_{G})$ such that
$\Sigma_{G} := V(G) \sqcup E(G)$ and $R_{G}$ contains all triples of the form 
$[v_{1}, e, v_{2}]$ for any edge $e \in E(G)$ with vertices $v_{1}$, $v_{2} \in V(G)$.
\end{Definition}

In some cases the cosmological polytopes $P(G)$ can be described explicitly. 
More precisely, the dimension of the cosmological polytope associated with a graph $G$ 
is equal to $|V(G)|+|E(G)|-1$. So the only 2-dimensional case is given be the graph 
which is an edge with two vertices. For this graph the corresponding cosmological polytope 
is just a triangle. Note that the same polytope is given by the Graev construction for 
the root system $A_2$. In fact, we have the isomorphism of the ternary relations 
$\mathcal{T}$\mbox{$\Big($
\includegraphics[width=0.05\textwidth]{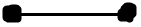}
$\Big)$}$\simeq \mathcal{T}(A_{2})$.
Next, the only 3-dimensional case is given by the graph which consists of two vertices 
connected by two parallel edges. The corresponding cosmological polytope is a prism 
coincided with the Graev polytope for $B_2$ because again we have the isomorphism of 
ternary relations 
$\mathcal{T}$\mbox{$\Big($\includegraphics[width=0.05\textwidth]{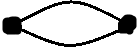}
$\Big)$}$\simeq \mathcal{T}(B_{2})$.

In contrast with Graev polytopes $P(\Omega)$, the cosmological polytopes $P(G)$ are much
better studied and there are general results about their structure (see \cite{B, J-KSV, 
BGJLS, KM}). For example, it is well-known that the facets of $P(G)$ are in one-to-one
correspondence with connected subgraphs of $G$.

\subsection{Finite metric spaces}

One can show that any element of the dual cone $P(A_{n-1})^{\vee}$ can be considered as 
a function over the set of edges of the complete graph $K_{n}$ satisfying some linear
inequalities which can be interpreted as triangle inequalities for a metric. In other
words, we have that $P(A_{n-1})^{\vee}$ is isomorphic to the so-called metric cone 
$M_{n} \subset \mathbb{R}^{n \choose 2}$ whose elements correspond to (semi-)metrics over 
the finite set consisted of $n$ points. The theory of finite metric spaces is well
developed (see \cite{DL}) and has many applications in the pure mathematics as well as in 
the different areas of applied mathematics, for example, in the computational biology (see
\cite{DMT}). At the same time this theory still has many open problems. One of them, which
is particularly interested for us, is the classification of extreme rays of the cone
$M_{n}$ (see \cite{BD, Av, Av2}). Due to the isomorphism $P(A_{n-1})^{\vee} \simeq 
M_{n}$ these extreme rays are in one-to-one correspondence with facets of the Graev 
polytope $P(A_{n-1})$. The complete description of these extreme rays is done only for 
$n \leq 7$ (see \cite{Gr}). However, some families of extreme rays of $M_{n}$ are known 
for arbitrary $n$. For example, there is the family of extreme rays generated by 
so-called cut-metrics (which are also sometimes called split-metrics or Hamming metrics) 
associated with cuts of the graph $K_{n}$. Another infinite family was constructed by Avis 
in \cite{Av} by considering graph metrics induced by subgraphs $G \subset K_{n}$ 
satisfying some properties.

\subsection{Main contributions}

\begin{enumerate}
\item In Section \ref{Simplicial posets} we generalize the cosmological construction of 
	a ternary relation. More precisely, instead of a graph which can be considered as 
	1-dimensional simplicial poset we consider 2-dimensional simplicial poset 
	$\mathcal{P}$ (see \cite{S1, S2}) and associate with it the ternary	relation 
	$\mathcal{T}(P)$.
	\begin{enumerate}	
		\item We prove that for any graph $G$ we have $P(G) = P(CG)$ where $CG$ is 
			a 2-dimensional simplicial poset called a cone over $G$
			(see Lemma \ref{ConeGraphIsomorphism}).

		\item Next, we show that $P(A_{n-1}) = P(\mathcal{K}_{n})$ where $\mathcal{K}_{n}$ 
			is the 2-skeleton of the face poset of the standard $(n-1)$-dimensional 
			simplex $\Delta_{n-1} \subset \mathbb{R}^{n}$ 
			(see Lemma \ref{AnIsomorphism}).

		\item Finally, for the root systems $D_{n}$ and $B_{n}$ we construct 2-dimensional
			simplicial posets $\overline{\mathcal{K}}_{n}$ and
			$\big( C\overline{\mathcal{K}}_{n} \big)^{(2)}$ such that 
			$P(D_{n}) = P(\overline{\mathcal{K}}_{n})$ and 
			$P(B_{n}) = P\Big( \big( C\overline{\mathcal{K}}_{n} \big)^{(2)} \Big)$
			(see Lemmas \ref{DnIsomorphism} and \ref{BnIsomorphism}).
	\end{enumerate}
\item We partially describe the facet structure of the ternary polytopes associated with
	2-dimensional simplicial posets in two ways.
	\begin{enumerate}
		\item In Section \ref{Minimal markings} we extend the technique of markings used
			in the theory of cosmological polytopes to the case of 2-dimensional
			simplicial posets (see Theorem \ref{CocycleMarkings} and its corollaries).

		\item In Section \ref{Extremal metrics} we introduce the notion of a metric on
			2-dimensional simplicial poset and generalize the result of Avis (see Theorems
			\ref{ExtremalMetrics} and \ref{ExtremalMetricsInDn}).
	\end{enumerate}
\end{enumerate} 

\section{Simplicial posets}\label{Simplicial posets}

\subsection{Preliminaries}
Let us remind the basic notions related to the theory of posets (see \cite{S2}). 
By definition, \textit{a poset} is a set equipped with a partial order $\leq$ which is 
reflexive, antisymmetric and transitive. A map $f: \mathcal{P}_{1} \longrightarrow 
\mathcal{P}_{2}$ between two posets is a map between corresponding sets which respects 
the orders in $\mathcal{P}_{1}$ and $\mathcal{P}_{2}$, i. e. $f(x) \leq f(y)$ in 
$\mathcal{P}_{2}$ for any $x \leq y$ in $\mathcal{P}_{1}$. In the usual way the notion of 
isomorphism is defined: two posets $\mathcal{P}_{1}$ and $\mathcal{P}_{2}$ are said to be 
isomorphic if there are two maps $f:\mathcal{P}_{1} \longrightarrow \mathcal{P}_{2}$ and 
$g: \mathcal{P}_{2} \longrightarrow \mathcal{P}_{1}$ such that 
$f \circ g = \text{id}_{\mathcal{P}_{2}}$ and $g \circ f = \text{id}_{\mathcal{P}_{1}}$.
The subposet consisted of elements $z$ of $\mathcal{P}$ which satisfy the inequality 
$x \leq z \leq y$ is called \textit{a closed interval} and denoted by 
$[x, y] \subset \mathcal{P}$. The elements $z$ of $\mathcal{P}$ satisfying the strict 
inequality $x < z < y$ form a subposet called \textit{an open interval} and denoted by 
$(x, y) \subset \mathcal{P}$. We will say that the element $x$ \textit{is covered by} 
the element $y$ or that $y$ \textit{covers} $x$ if we have that $x < y$ and $(x, y)$ 
is empty. A poset $\mathcal{P}$ is called graded if $\mathcal{P}$ is equipped with 
a rank function $\rho : \mathcal{P} \longrightarrow \mathbb{Z}$ satisfying the inequality 
$\rho(x) < \rho(y)$ for any elements $x, y \in \mathcal{P}$ such that $x < y$, 
and the equality $\rho(y) = \rho(x) + 1$ for all $x, y \in \mathcal{P}$ such that $y$
covers $x$. As an example of a graded poset consider the family of all subsets of 
the fixed set $\{1, .., n\}$ with the partial order induced by inclusion of subsets. 
This poset is usually denoted by $\mathcal{B}_{n}$ and being equipped with three 
operations $\land$, $\lor$ and $\neg$ given by the intersection, union and complement 
operations over sets it becomes so-called \textit{boolean algebra}. The rank function of 
$\mathcal{B}_{n}$ is defined as $\rho(S) = |S|$ for any subset $S \in \mathcal{B}_{n}$. 
Note that the empty set $\emptyset$ is also considered as an element of $\mathcal{B}_{n}$ 
with $\rho(\emptyset) = 0$. If a poset has an unique minimal element then it is usually 
denoted by $\hat{0}$, while the maximal element if it exists is denoted by $\hat{1}$. For
example, for $\mathcal{B}_{n}$ we have that $\hat{0} = \emptyset$ and 
$\hat{1} = \{1, .., n\}$.

Any simplicial complex $\mathcal{C}$ determines the so-called \textit{face poset} of 
$\mathcal{C}$ whose elements are faces of $\mathcal{C}$ and extra element $\hat{0}$ 
corresponding to empty face, while the partial order is induced by inclusion of faces. 
It is easy to see that $\mathcal{B}_{n}$ is actually the face poset of the standard 
$(n-1)$-dimensional simplex $\Delta_{n-1} \subset \mathbb{R}^{n}$. It gives motivation to 
introduce the following definition. A poset $\mathcal{P}$ with unique minimal element $\hat{0}$ 
is called \textit{a simplicial poset} if any closed interval $[\hat{0}, x] \subset 
\mathcal{P}$ is isomorphic to the boolean algebra $\mathcal{B}_{n}$ for some $n \geq 0$ 
depending on the element $x \in \mathcal{P}$. We will call elements of simplicial poset 
$\mathcal{P}$ \textit{simplices} in analogy with usual simplicial complexes. If two 
simplices $x, y \in \mathcal{P}$ are in the relation $x < y$ then we say that $x$ is 
\textit{a face} of $y$. If the simplex $x$ is covered by $y$ we say that $x$ is 
\textit{a facet} of $y$. Any simlicial poset is graded with the rank function $\rho$ 
satisfying the relation $[0, x] \simeq \mathcal{B}_{\rho(x)}$. The number $\rho(x)-1$ 
associated with any simplex $x \in \mathcal{P}$ is called \textit{the dimension} of $x$. 
The dimension of the simplicial poset $\mathcal{P}$ itself denoted by 
$\text{dim}~\mathcal{P}$ is the maximum of dimensions of all simplices of $\mathcal{P}$. 
We will use the notation $\mathcal{P}(k)$ for the subset of simplices $x \in \mathcal{P}$ 
whose dimension is equal to $k$, i. e. $\rho(x)-1 = k$. Also we call the subposet of 
$\mathcal{P}$ consisted of simplices whose dimension is less or equal to $k$ by 
\textit{the $k$-skeleton} of $\mathcal{P}$ and denote it by 
$\mathcal{P}^{(k)} \subset \mathcal{P}$. A simplicial poset $\mathcal{P}$ is said to be 
\textit{pure} if all maximal simplices of $\mathcal{P}$ have the same dimension 
equal to $\text{dim}~\mathcal{P}$.

Having two posets $\mathcal{P}_{1}$ and $\mathcal{P}_{2}$ we can construct a new one 
$\mathcal{P}_{1} \times \mathcal{P}_{2}$ called the \textit{direct product of}
$\mathcal{P}_{1}$ and $\mathcal{P}_{2}$ which is the direct product of underlying sets 
with the partial order satisfying the condition: $(x_{1}, y_{1}) \leq (x_{2}, y_{2})$ 
if and only if $x_{1} \leq_{\mathcal{P}_{1}} x_{2}$ and $y_{1} \leq_{\mathcal{P}_{2}} 
y_{2}$. One can check that $\mathcal{B}_{n} \times \mathcal{B}_{m} \simeq 
\mathcal{B}_{n+m-1}$ and the direct product of two simplicial posets is again simplicial. 
For any simplicial poset $\mathcal{P}$ define \textit{the cone over} $\mathcal{P}$ as 
the direct product $C\mathcal{P} := \mathcal{P} \times \mathcal{B}_{1}$. Note that we have 
the equality $\text{dim}~C\mathcal{P} = \text{dim}~\mathcal{P} + 1$. We will call 
the additional vertex $(\hat{0}_{\mathcal{P}}, \hat{1}_{\mathcal{B}_{1}})$ of 
$C\mathcal{P}$ by \textit{the apex} of the cone $C\mathcal{P}$.

One can show that the face poset of any simplicial complex is a simplicial poset. 
On the other hand, it is not true that any simplicial poset is isomorphic to the face 
poset of simplicial complex. However, any simplicial poset can be represented as the face 
poset of a regular CW-complex which differs from simplicial complex only by the fact that 
it allows two different simplices to have more than one common facets. Note that 
1-dimensional simplicial poset can be considered just as a graph which may have multiple 
edges, but loops are prohibited. In particular, it means that the 1-skeleton 
$\mathcal{P}^{(1)}$ of any simplicial poset $\mathcal{P}$ is a graph. This observation 
will be used throughout the paper. 

\subsection{Ternary relations associated with 2-dimensional simplicial posets}
Further we will consider only 2-dimensional simplicial posets. We will call 0-simplices by 
\textit{vertices}, 1-simplices by \textit{edges} and 2-simplices by \textit{triangles}. 
For any 2-dimensional simplicial poset we can construct a ternary relation as follows:

\begin{Definition}\label{PosetTernaryRelation}
	Let $\mathcal{P}$ be a 2-dimensional simplicial poset. A ternary relation 
	$\mathcal{T}(\mathcal{P})$ is a pair $(S_{\mathcal{P}}, R_{\mathcal{P}})$ such that 
	$S_{\mathcal{P}} = \mathcal{P}(1)$ and $R_{\mathcal{P}}$ contains all the triples
	$[e_1, e_2, e_3]$ for any triangle $\Delta \in \mathcal{P}(2)$ with facets $e_1$,
	$e_2$, $e_3$.
\end{Definition}

\noindent As before we can consider the ternary polytope associated with this ternary
relation which we denote py $P(\mathcal{P}) := P\big( \mathcal{T}(\mathcal{P}) \big)$.
Definition \ref{PosetTernaryRelation} extends the construction of a ternary relation by 
a graph in the following sense. 

\begin{Lem}\label{ConeGraphIsomorphism}
	For any graph $G$ considered as 1-dimensional simplicial poset we have the isomorphism
	of ternary relations $\mathcal{T}(CG) \simeq \mathcal{T}(G)$ implying $P(CG) = P(G)$.	
\end{Lem}
\begin{proof}
	By construction $\Sigma_{CG}$ coincides with the set $CG(1)$ of edges of $CG$. 
	Any edge of $CG$ has either the form $(e, \hat{0}_{\mathcal{B}_{1}})$ for 
	$e \in E(G)$, or $(v, \hat{1}_{\mathcal{B}_{1}})$ for $v \in V(G)$. 
	Taking into account that $\Sigma_{G} = E(G) \sqcup V(G)$ we obtain the canonical
	bijection $\phi: \Sigma_{CG} \DistTo \Sigma_{G}$. The triangles of $CG$ have the form 
	$(e, \hat{1}_{\mathcal{B}_{1}})$ for any $e \in E(G)$. Each of them has the following 
	facets: $(e, \hat{0}_{\mathcal{B}_{1}})$ and $(v_1, \hat{1}_{\mathcal{B}_{1}})$, 
	$(v_2, \hat{1}_{\mathcal{B}_{1}})$ where $v_1, v_2 \in V(G)$ are end-vertices of 
	the edge $e \in E(G)$. Therefore, the set $R_{CG}$ consists of the triples of the form 
	$[(e, \hat{0}_{\mathcal{B}_{1}}), (v_{1}, \hat{1}_{\mathcal{B}_{1}}), 
	(v_{2}, \hat{1}_{\mathcal{B}_{1}})]$ for any edge $e \in E(G)$ with end-vertices 
	$v_{1}$, $v_{2} \in V(G)$ which are in one-to-one correspondence with the triples 
	$[e, v_1, v_2]$ of $R_{G}$ under the map $\phi$.
\end{proof}
\noindent
Cones over some graphs see in Fig. \ref{Cones_over_graphs}.

\begin{figure}[t]
\includegraphics[width=.8\textwidth]{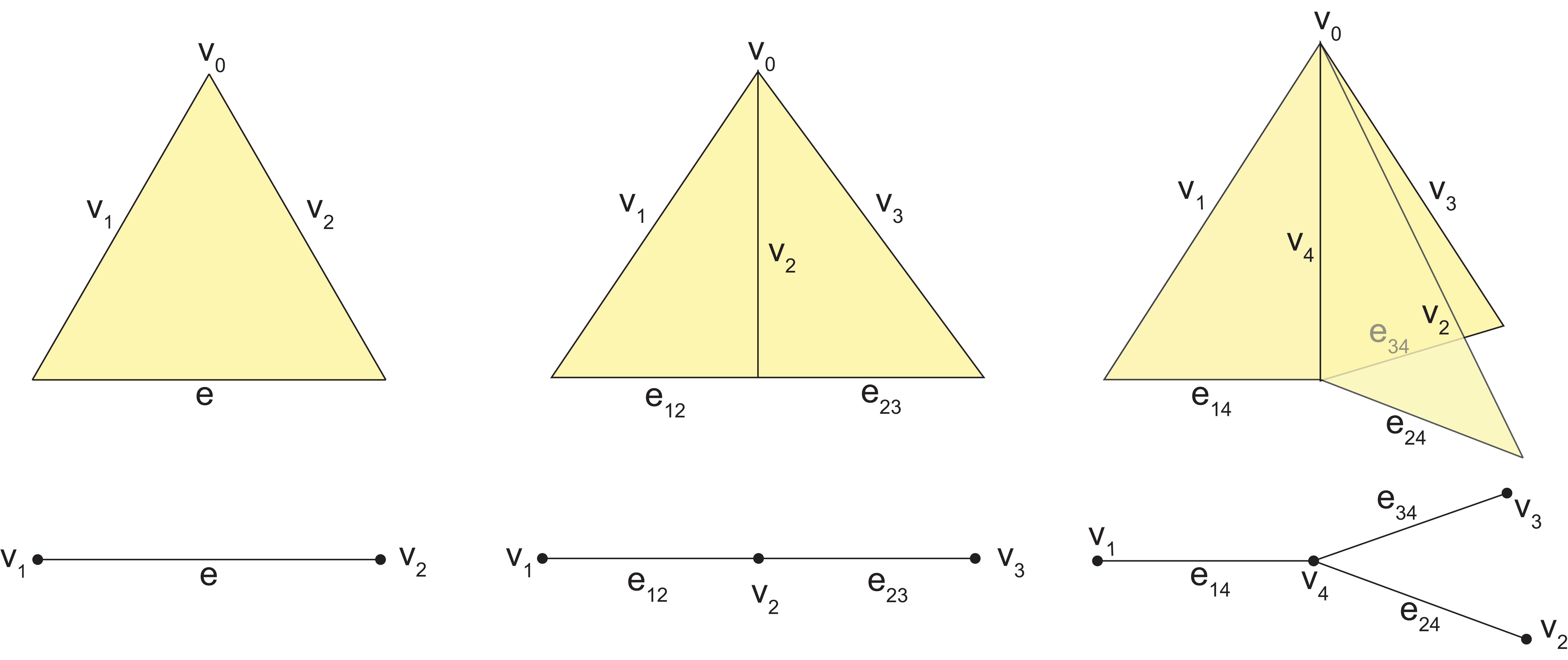}
	\caption{Cones over some graphs. Note that the set of edges of the cone is the union 
	of edges and vertices of the initial graph. The apex of each cone is denoted 
	by $v_{0}$.}\label{Cones_over_graphs}
\end{figure}

Ternary relations induced by some root systems can also be obtained as ternary relations
of simplicial posets. Denote the 2-skeleton of the boolean algebra $\mathcal{B}_{n}$ by
$\mathcal{K}_{n}$, i. e. $\mathcal{K}_{n} := \mathcal{B}_{n}^{(2)}$. Geometrically,
$\mathcal{K}_{n}$ can be seen as the face poset of the 2-skeleton of the standard 
$(n-1)$-dimensional simplex $\Delta_{n-1} \subset \mathbb{R}^{n}$. Note that 
the 1-skeleton of $\mathcal{K}_{n}$ is isomorphic to the complete graph $K_{n}$, i. e.
$\mathcal{K}_{n}^{(1)} \simeq K_{n}$, this is the reason for the notation 
$\mathcal{K}_{n}$ is used. Also we have that triangles of $\mathcal{K}_{n}$ are in 
one-to-one correspondence with 3-cycles of the graph $\mathcal{K}_{n}^{(1)}$.

\begin{Lem}\label{AnIsomorphism}
	We have the isomorphism $\mathcal{T}(\mathcal{K}_{n}) \simeq \mathcal{T}(A_{n-1})$ 
	implying $P(\mathcal{K}_{n}) = P(A_{n-1})$.
\end{Lem}
\begin{proof}
By construction, the vertices of $\mathcal{K}_{n}$ can be represented as sets $v_{1} := 
\{ 1 \}$, .., $v_{n} := \{ n \}$, the edges as 2-element sets $e_{ij} := \{ i, j \}$ 
where $1 \leq i \neq j \leq n$, finally, the triangles as 3-element sets $\Delta_{ijk}:= 
\{ i, j, k\}$ for any $1 \leq i \neq j \neq k \leq n$. Note that the notations $e_{ij}$
and $\Delta_{ijk}$ are invariant under permutations of indicies. The triples in 
the ternary relation $\mathcal{T}(\mathcal{K}_{n})$ have the form 
$[ e_{ij}, e_{jk}, e_{ik}]$ for $1 \leq i < j < k \leq n$. On the other hand, 
the root system $A_{n-1}$ is the configuration of vectors $\big\{ \pm (\varepsilon_{i} -
\varepsilon_{j}) \ \mid \ 1 \leq i \neq j \leq n \big\} \subset \mathbb{R}^{n}$ where 
$\{\varepsilon_{i} \ \mid \ 1 \leq i \leq n \}$ is the standard basis of the vector space 
$\mathbb{R}^{n}$. The triples of the ternary relation $\mathcal{T}(A_{n})$ are 
the following: $[\varepsilon_{i} - \varepsilon_{j}, \varepsilon_{j} - \varepsilon_{k}, 
\varepsilon_{i} - \varepsilon_{k}]$ for any $1 \leq i < j < k \leq n$. Therefore, we have 
the natural correspondence $e_{ij} \leftrightarrow \varepsilon_{i} - \varepsilon_{j}$, 
$1 \leq i < j \leq n$, which gives the isomorphism of the ternary relations.
\end{proof}

Now construct the extension $\overline{\mathcal{K}}_{n}$ of $\mathcal{K}_{n}$ as follows. 
The vertices of $\overline{\mathcal{K}}_{n}$ are the same as for $\mathcal{K}_{n}$, 
i. e. $\overline{\mathcal{K}}_{n}(0) = \mathcal{K}_{n}(0) = \{ v_{1}, .., v_{n}\}$. 
On the other hand, $\mathcal{K}_{n}(1) \subsetneq \overline{\mathcal{K}}_{n}(1)$ and 
$\mathcal{K}_{n}(2) \subsetneq \overline{\mathcal{K}}_{n}(2)$. More precisely, along with 
edges of the form $e_{ij} \in \mathcal{K}_{n}(1) \subset \overline{\mathcal{K}}_{n}(1)$ 
the poset $\overline{\mathcal{K}}_{n}$ has edges which we denote by $e^{i}_{j}=e^{j}_{i}$, 
$1 \leq i \neq j \leq n$. The facets of the edge $e^{i}_{j}$ are the same as for the edge 
$e_{ij}$, namely, the vertices $v_{i}$ and $v_{j}$. Next, along with triangles 
$\Delta_{ijk} \in \mathcal{K}_{n}(2) \subset \overline{\mathcal{K}}_{n}(2)$ the poset
$\overline{\mathcal{K}}_{n}$ contains triangles denoted by $\Delta^{i}_{jk} =
\Delta^{i}_{kj}$, $1 \leq i \neq j \neq k \leq n$. The 0-dimensional faces of the triangle 
$\Delta^{i}_{jk}$ are the vertices $v_{i}$, $v_{j}$ and $v_{k}$, while the facets are 
the following edges: $e^{i}_{j}$, $e^{i}_{k}$ and $e_{jk}$. Summarizing, the poset
$\overline{\mathcal{K}}_{n}$ is defined as follows:
\begin{gather}
	\overline{\mathcal{K}}_{n}(0) = \{v_{1}, .., v_{n}\}, \ \ \
	\overline{\mathcal{K}}_{n}(1) = \{e_{ij} \ | \ 1 \leq i \neq j \leq n \} \sqcup
		\{ e^{i}_{j} \ | \ 1 \leq i \neq j \leq n \}, \notag
	\\
	\overline{\mathcal{K}}_{n}(2) = \{ \Delta_{ijk} \ | \ 1 \leq i \neq j \neq k \leq n \}
		\sqcup \{ \Delta^{i}_{jk} \ | \ 1 \leq i \neq j \neq k \leq n\}, \notag
	\\
	\text{facets of } e_{ij} \text{ and } e^{i}_{j} \text{ are } v_{i}, v_{j}, \notag
	\\
	\text{facets of } \Delta_{ijk} \text{ are } e_{ij}, e_{jk}, e_{ik}, \ \ \
	\text{facets of } \Delta^{i}_{jk} \text{ are } e^{i}_{j}, e^{i}_{k}, e_{jk}. \notag
\end{gather}
Straightforward computations show that any closed interval $[\hat{0}, x]$ in 
$\overline{\mathcal{K}}_{n}$ is isomorphic to boolean algebra, so
$\overline{\mathcal{K}}_{n}$ is a simplicial poset. However, in contrast with 
$\mathcal{K}_{n}$, it is already not the face poset of a simplicial complex. To see this 
consider the 1-skeleton of $\overline{\mathcal{K}}_{n}$ and denote it by
$\overline{K}_{n}:=\overline{\mathcal{K}}^{(1)}$. Note that the graph $\overline{K}_{n}$
contains exactly two edges, $e_{ij}$ and $e^{i}_{j}$, for each pair of vertices $v_{i},
v_{j}$. So $\overline{K}_{n}$ is not the face poset of a simplicial complex which implies 
the same property for $\overline{\mathcal{K}}_{n}$. The graph $\overline{K}_{n}$ can be 
seen as "doubling" of the complete graph $K_{n}$. Also note that the triangles of
$\overline{\mathcal{K}}_{n}$ are given by all 3-cycles in $\overline{K}_{n}$ which have
even number of edges of the form $e^{i}_{j}$.

\begin{Lem}\label{DnIsomorphism}
	We have the isomorphism $\mathcal{T}(\overline{\mathcal{K}}_{n}) \simeq 
	\mathcal{T}(D_{n})$ implying $P(\overline{\mathcal{K}}_{n}) = P(D_{n})$.
\end{Lem}
\begin{proof}
By definition, $D_{n}$ is the configuration of vectors $\big\{ \pm (\varepsilon_{i} -
\varepsilon_{j}), \pm(\varepsilon_{i}+\varepsilon_{j}) \ \mid \ 1 \leq i<j \leq n \big\}
\subset \mathbb{R}^{n}$, where $\{ \varepsilon_{i} \ \mid \ 1 \leq i \leq n \}$ is 
the standard basis in the vector space $\mathbb{R}^{n}$. The triples of the ternary 
relation $\mathcal{T}(D_{n})$ have the following forms: $[\varepsilon_{i}-\varepsilon_{j},
\varepsilon_{j} - \varepsilon_{k}, \varepsilon_{i} - \varepsilon_{k}]$ 
for $1 \leq i < j < k \leq n$, and $[\varepsilon_{i} + \varepsilon_{j}, 
\varepsilon_{i} + \varepsilon_{k}, \varepsilon_{j} - \varepsilon_{k}]$, 
$[\varepsilon_{i} + \varepsilon_{j}, \varepsilon_{j} + \varepsilon_{k},
\varepsilon_{i} - \varepsilon_{k}]$, 
$[\varepsilon_{i} + \varepsilon_{k}, \varepsilon_{j} + \varepsilon_{k},
\varepsilon_{i} - \varepsilon_{j}]$ for $1 \leq i < j < k \leq n$.
Consider the correspondence similar to the previous lemma:
$e_{ij} \leftrightarrow \varepsilon_{i} - \varepsilon_{j}$,
$e^{i}_{j} \leftrightarrow \varepsilon_{i} + \varepsilon_{j}$,
$1 \leq i < j \leq n$.
From this correspondence we see that the first form of triples in $\mathcal{T}(D_{n})$ 
is given by triangle $\Delta_{ijk}$, while three other forms are given by the triangles
$\Delta^{i}_{jk}$, $\Delta^{j}_{ik}$, $\Delta^{k}_{ij}$, respectively. It gives us 
the isomorphism between considered ternary relations.
\end{proof}

Finally, consider the simplicial poset $\big(C\overline{\mathcal{K}}_{n}\big)^{(2)}$ 
which is the 2-skeleton of the cone over $\overline{\mathcal{K}}_{n}$. Along with vertices, 
edges and triangles of $\overline{\mathcal{K}}_{n}$ it contains for each 
$i \in \{1, .., n\}$ the edge $e_{i}$ which connects the vertex 
$v_{i} \in \overline{\mathcal{K}}_{n}(0)$ with the apex 
$v_{0} \in \big(C\overline{\mathcal{K}_{n}}\big)^{(2)}(0)=C\overline{\mathcal{K}}_{n}(0)$ 
and the following triangles: $\Delta_{ij}=\Delta_{ji}$ 
with facets $e_{i}, e_{ij}, e_{j}$, and $\Delta^{i}_{j} = \Delta^{j}_{i}$ 
with facets $e_{i}, e^{i}_{j}, e_{j}$ for any $1 \leq i \neq j \leq n$.

\begin{Lem}\label{BnIsomorphism}
	We have the isomorphism 
	$\mathcal{T}\Big( \big(C\overline{\mathcal{K}}_{n}\big)^{(2)} \Big) \simeq 
	\mathcal{T}(B_{n})$ implying 
	$P\Big( \big( C\overline{\mathcal{K}}_{n} \big)^{(2)} \Big) = P(B_{n})$.
\end{Lem}
\begin{proof}
The root system $B_{n}$ is the configuration of vectors $\big\{ \pm(\varepsilon_{i} -
\varepsilon_{j}), \pm (\varepsilon_{i} + \varepsilon_{j}) \ \mid \ 1 \leq i < j \leq n
\big\} \sqcup \big\{ \pm \varepsilon_{i} \ \mid 1 \leq i \leq n \big\} \subset
\mathbb{R}^{n}$, where $\{ \varepsilon_{i} \ | \ 1 \leq i \leq n \}$ is again 
the standard basis of the vector space $\mathbb{R}^{n}$.
The triples in $\mathcal{T}(B_{n})$ are the following:
$[\varepsilon_{i}-\varepsilon_{j}, \varepsilon_{j}-\varepsilon_{k}, 
\varepsilon_{i}-\varepsilon_{k}]$,
$[\varepsilon_{i}+\varepsilon_{j}, \varepsilon_{i}+\varepsilon_{k},
\varepsilon_{j}-\varepsilon_{k}]$,
$[\varepsilon_{i}+\varepsilon_{j}, \varepsilon_{j}+\varepsilon_{k},
\varepsilon_{i}-\varepsilon_{k}]$,
$[\varepsilon_{i}+\varepsilon_{k}, \varepsilon_{j}+\varepsilon_{k},
\varepsilon_{i}-\varepsilon_{j}]$,
$[\varepsilon_{i}, \varepsilon_{j}, \varepsilon_{i}-\varepsilon_{j}]$,
$[\varepsilon_{i}, \varepsilon_{j}, \varepsilon_{i}+\varepsilon_{j}]$
for $1 \leq i < j < k \leq n$.
Consider the correspondence:
$e_{ij} \leftrightarrow \varepsilon_{i}-\varepsilon_{j}$,
$e^{i}_{j} \leftrightarrow \varepsilon_{i}+\varepsilon_{j}$,
$\varepsilon_{i} \leftrightarrow e_{i}$.
We see that in terms of this correspondence the triples above correspond to the triangles
$\Delta_{ijk}$, $\Delta^{i}_{jk}$, $\Delta^{j}_{ik}$, $\Delta^{k}_{ij}$, $\Delta_{ij}$,
$\Delta^{i}_{j}$ of $\big(C \overline{\mathcal{K}}_{n}\big)^{(2)}$.
Therefore, we have the isomorphism of ternary relations.
\end{proof}

\section{Minimal markings}\label{Minimal markings}

Note that for any ternary relation $\mathcal{T} = (\Sigma, R)$ the affine hull 
$\text{aff}(P(\mathcal{T})) = \Big\{ \sum\limits_{i=1}^{n} x^{i} = 1 \Big\} \subset
\mathbb{R}^{n}$, $n = \mid \Sigma \mid$ does not contain the origin. So the facets of
$P(\mathcal{T})$ uniquely correspond to extreme rays of its dual cone 
$P(\mathcal{T})^{\vee}$ which we define as follows:
$P(\mathcal{T})^{\vee} = \big\{ f \in (\mathbb{R}^{n})^{\vee} \ | \ f(v) \geq 0 \ \forall v
\in P(\mathcal{T}) \big\}$. Let $\textbf{1}^{1}$, .., $\textbf{1}^{n}$ be the basis of 
$(\mathbb{R}^{n})^{\vee}$ which is dual to the basis 
$\textbf{1}_{1}$, .., $\textbf{1}_{n}$ from Definition \ref{TernaryPolytope} and 
$x_1$, .., $x_{n}$ are corresponding coordinates of $(\mathbb{R}^{n})^{\vee}$. 
Then the dual cone $P(\mathcal{T})^{\vee}$ has the expression
\begin{equation}\label{ConeDefinition}
	P( \mathcal{T} )^{\vee} = \bigcap\limits_{[i, j, k] \in R} 
		\Big\{ x_{i} + x_{j} - x_{k} \geq 0, \ 
		x_{i} - x_{j} + x_{k} \geq	0, \ 
		-x_{i} + x_{j} + x_{k} \geq 0 \Big\}.
\end{equation}
For the case of a ternary relation $\mathcal{T} = \mathcal{T}(\mathcal{P})$ induced by a
2-dimensional simplicial poset $\mathcal{P}$ the elements of the dual cone
$P(\mathcal{P})^{\vee}$ can be considered as functions over the set $\mathcal{P}(1)$ of
edges of $\mathcal{P}$.

An extreme ray of $P(\mathcal{T})^{\vee}$ can be determined by the choice of subcollection 
of inequalities from (\ref{ConeDefinition}) which become strict, while other vanish.
In the case of cosmological polytopes, a subcollection of inequalities corresponds to 
a marking of the graph which is a configuration of marks placed on the edges of the graph. 
For giving edge the mark can be placed in three possible locations: in the middle of 
the edge or close to one from the end-nodes of the edge. This approach allows to study 
the structure of facets and, more generally, faces of cosmological polytopes in terms of 
combinatorics of the corresponding graphs. In particular, it is proved in \cite{A-HBP} 
using this approach that all facets of the cosmological polytope determined by 
the graph $G$ are in one-to-one correspondence with connected subgraphs of $G$. 
In this section we apply this approach to study the facet structure of ternary polytopes 
of the form $P(\mathcal{P})$.

Suppose that we have a 2-dimensional simplicial poset $\mathcal{P}$. We will call a pair 
$(\Delta, [e_1, e_2])$ consisted of the triangle $\Delta \in \mathcal{P}(2)$ and 
the unordered pair of two distinct edges $e_1, e_2 \in \mathcal{P}(1)$ of $\Delta$ 
by \textit{a corner} of $\mathcal{P}$. From this definition it follows that any 
triangle of $\mathcal{P}$ has exactly 3 corners. Also note that each corner of 
$\mathcal{P}$ corresponds to particular inequality defining the cone 
$P(\mathcal{P})^{\vee}$ in the following way. Let $e_1, e_2, e_3$ be the facets of 
the triangle $\Delta \in \mathcal{P}(2)$, then the corner $(\Delta, [e_1, e_2])$ 
corresponds to the inequality $x_1 + x_2 - x_3 \geq 0$ where $x_1, x_2, x_3$ are 
the corresponding dual variables.

\begin{figure}[t]
\includegraphics[width=.9\textwidth]{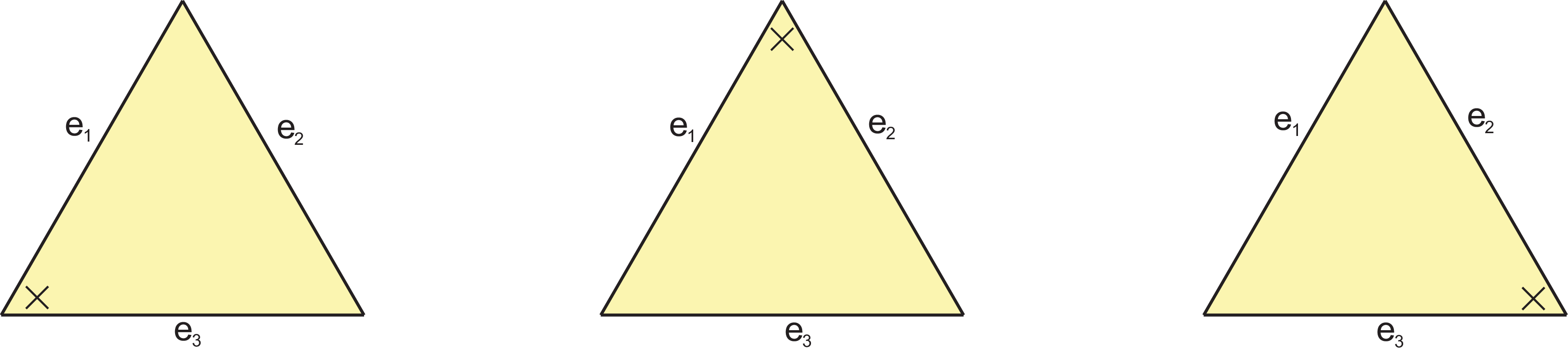}
	\caption{Possible locations of marks in the triangle. Note that these three markings 
	are the only minimal feasible markings of this poset.}
	\label{Possible_locations_for_marks}
\end{figure}

We will call collection of different corners of $\mathcal{P}$ \textit{a marking} 
of $\mathcal{P}$. For fixed marking $M$ we will say that given corner \textit{is marked
in} $M$ or mark is placed in that corner if this corner belongs to the marking 
(see Fig. \ref{Possible_locations_for_marks}). 
We will consider the subset $P(\mathcal{P})^{\vee}_{M} \subset P(\mathcal{P})^{\vee}$ 
consisted of those covectors of $P(\mathcal{P})^{\vee}$ which make the inequalities 
corresponding to the marked corners strict, while all other inequalities are vanishing. 
For example, the empty marking corresponds to the zero subspace, i. e. 
$P(\mathcal{P})^{\vee}_{\emptyset} = \{ 0 \} \subset P(\mathcal{P})^{\vee}$.
On the other hand, the maximal marking containing all corners of $\mathcal{P}$ corresponds 
to the interior set of the cone $P(\mathcal{P})^{\vee}$.

There are markings $M$ for which the subsets $P(\mathcal{P})^{\vee}_{M}$ are empty. 
We will call a marking $M$ \textit{feasible} if it defines non-empty subset
$P(\mathcal{P})^{\vee}_{M}$ and \textit{infeasible}, otherwise. One can easily provide 
a necessary condition for marking to be feasible. More precisely, take arbitrary edge 
$e \in \mathcal{P}(1)$ and consider two situation. The first one is that the marking $M$ 
contains a corner of the form $(\Delta, [e, e'])$ for some triangle 
$\Delta \in \mathcal{P}(2)$ with edges $e$, $e'$ and $e''$. In this case we say that the
triangle $\Delta$ \textit{is marked in $M$ along its edge} $e$. By definition, it means 
that the covectors in $P(\mathcal{P})^{\vee}_{M}$ satisfy the strict inequality 
$x_{e}+x_{e'}-x_{e''} > 0$. Taking into account the inequality 
$x_{e}-x_{e'}+x_{e''} \geq 0$ which is true in any case we obtain that $x_{e} > 0$. 
Now consider the second situation. Suppose that $\mathcal{P}$ contains some triangle
$\Delta$ with edges $e$, $e'$, $e''$ which is not marked along $e$. It means that
corners $(\Delta, [e, e'])$ and $(\Delta, [e, e''])$ are not marked, so we have two
equalities: $x_{e}+x_{e'}-x_{e''}=0$ and $x_{e}-x_{e'}+x_{e''}=0$. These equalities imply
that $x_{e}=0$. Therefore, we see that two discussed situations contradict each other.
Combinatorially, it means that for any feasible marking and for any edge $e$ of
$\mathcal{P}$ we have either all triangles having $e$ as a facet are marked along the edge
$e$ or none of them is marked along $e$. If for giving marking this condition is satisfied
we will call it \textit{locally feasible} marking
(see in Fig. \ref{Not_locally_admissible_markings} the typical examples of non locally
feasible markings). Also we wil say that the edge $e$ \textit{is marked in} $M$ if all 
triangles having it as a facet are marked in $M$ along $e$. Note that this condition is 
not sufficient in general and there are locally feasible markings which are infeasible. 
For giving simplicial poset $\mathcal{P}$ we denote the set of all markings of 
$\mathcal{P}$ by $\mathcal{M}(\mathcal{P})$, the set of feasible markings by 
$\mathcal{M}_{\text{f}}(\mathcal{P})$ and the set of locally feasible markings by 
$\mathcal{M}_{\text{loc}}(\mathcal{P})$. We have the following relations between 
these sets: 
$\mathcal{M}_{\text{f}}(\mathcal{P}) \subset \mathcal{M}_{\text{loc}}(\mathcal{P})
\subset \mathcal{M}(\mathcal{P})$.

\begin{figure}[t]
\includegraphics[width=.9\textwidth]{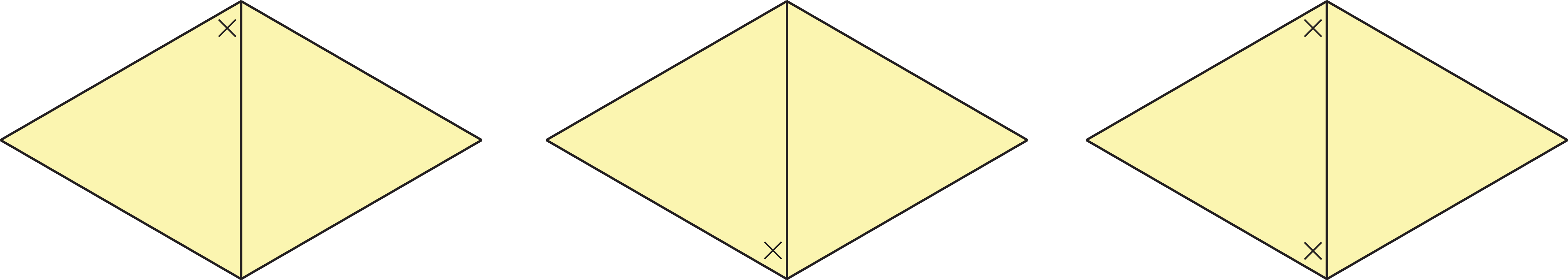}
\caption{Not locally feasible markings.}\label{Not_locally_admissible_markings}
\end{figure}

The set of all markings $\mathcal{M}(\mathcal{P})$ as well as its subsets
$\mathcal{M}_{\text{f}}(\mathcal{P})$ and $\mathcal{M}_{\text{loc}}(\mathcal{P})$ carry 
the partial order induced by the inlcusion of markings, namely, 
$M_{1} \leq M_{2}$ if and only if $M_{1} \subset M_{2}$. For each of
$\mathcal{M}_{\text{f}}(\mathcal{P})$ and $\mathcal{M}_{\text{loc}}(\mathcal{P})$ we
denote the subset of minimal elements by
$\mathcal{M}_{\text{f}}^{\text{min}}(\mathcal{P})$ and
$\mathcal{M}_{\text{loc}}^{\text{min}}(\mathcal{P})$, respectively.
One can show that markings from $\mathcal{M}_{\text{f}}^{\text{min}}(\mathcal{P})$ are in
one-to-one correspondence with extreme rays of the cone $P(\mathcal{P})^{\vee}$. 
In the cosmological case, i. e. for simplicial posets of the form $CG$, it is proved 
in \cite{A-HBP} that
$\mathcal{M}_{\text{f}}^{\text{min}}(CG) = \mathcal{M}_{\text{loc}}^{\text{min}}(CG)$,
and, moreover, the markings from $\mathcal{M}_{\text{loc}}^{\text{min}}(CG)$ are
in one-to-one correspondence with connected subgraphs of the graph $G$ 
(see in Fig. \ref{Minimal_admissible_markings} all minimal feasible markings of $CG$ where 
$G$ is the line graph consisted of two subsequent edges). However, in general we have
neither 
$\mathcal{M}_{\text{loc}}^{\text{min}}(\mathcal{P}) \subset 
\mathcal{M}_{\text{f}}^{\text{min}}(\mathcal{P})$
nor
$\mathcal{M}_{\text{f}}^{\text{min}}(\mathcal{P}) 
\subset \mathcal{M}_{\text{loc}}^{\text{min}}(\mathcal{P})$. One of the simplest examples 
for which it happens is the poset $\mathcal{K}_{5}$ associated with 
the root system $A_{4}$.

\begin{figure}[t]
\includegraphics[width=.9\textwidth]{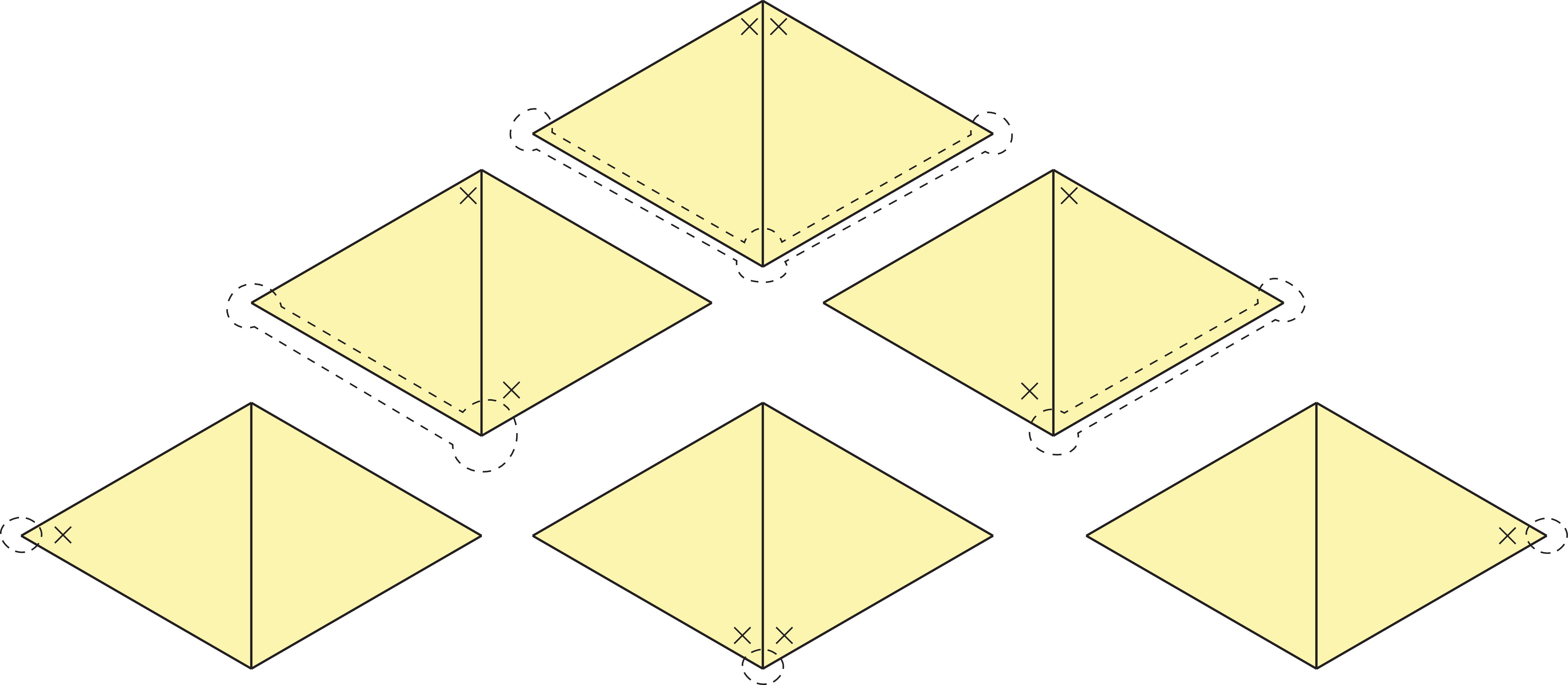}
\caption{All minimal feasible markings of the cone over the line graph with 2 subsequent 
	edges. The corresponding connected subgraphs of the initial graph are 
	highlighted by dashed line according to the theory of cosmological polytopes.}
	\label{Minimal_admissible_markings}
\end{figure}

Nevertheless, it is possible to extend the equality 
$\mathcal{M}_{\text{f}}^{\text{min}}(CG) = \mathcal{M}_{\text{loc}}^{\text{min}}(CG)$
to the case of of more general ternary polytopes if we restrict our consideration to some 
subclass of locally feasible markings. More precisely, consider the subset of markings 
$\mathcal{M}^{1}(\mathcal{P}) \subset \mathcal{M}(\mathcal{P})$ which have no more than 
one marked corner in each triangle. We will use the following notations:
$\mathcal{M}_{\text{loc}/\text{f}}^{1}(\mathcal{P}) := 
\mathcal{M}_{\text{loc}/\text{f}}(\mathcal{P}) \cap \mathcal{M}^{1}(\mathcal{P})$
and
$\mathcal{M}_{\text{loc}/\text{f}}^{1, \text{min}}(\mathcal{P}) := 
\mathcal{M}_{\text{loc}/\text{f}}^{\text{min}}(\mathcal{P}) \cap 
\mathcal{M}^{1}(\mathcal{P})$.

\begin{Lem}\label{OneAdmissibleMarkings}
	$\mathcal{M}_{\text{f}}^{1}(\mathcal{P}) = \mathcal{M}_{\text{loc}}^{1}(\mathcal{P})$ 
	and, consequently,
	$\mathcal{M}_{\text{f}}^{1, \text{min}}(\mathcal{P}) = 
		\mathcal{M}_{\text{loc}}^{1, \text{min}}(\mathcal{P})$.
\end{Lem}
\begin{proof}
Since $\mathcal{M}_{\text{f}}(\mathcal{P}) \subset \mathcal{M}_{\text{loc}}(\mathcal{P})$ 
we have that 
$\mathcal{M}_{\text{f}}^{1}(\mathcal{P}) \subset 
\mathcal{M}_{\text{loc}}^{1}(\mathcal{P})$. 
On the other hand, let $M$ be a marking from $\mathcal{M}_{\text{loc}}^{1}(\mathcal{P})$. 
Consider the covector $f(M) \in P(\mathcal{P})^{\vee}$ which is equal to 1 for edges
marked in $M$, while for any other edge it vanishes. It is easy to see that this covector 
satisfies all inequalities which determine the subset $P(\mathcal{P})^{\vee}_{M} \subset 
P(\mathcal{P})^{\vee}$, so we have $f(M) \in P(\mathcal{P})^{\vee}_{M} \neq \emptyset$. 
Therefore, we obtain that $M \in \mathcal{M}_{\text{f}}^{1}(\mathcal{P})$ and, 
consequently, $\mathcal{M}_{\text{loc}}^{1}(\mathcal{P}) \subset 
\mathcal{M}_{\text{f}}^{1}(\mathcal{P})$, so 
$\mathcal{M}_{\text{f}}^{1}(\mathcal{P}) = \mathcal{M}_{\text{loc}}^{1}(\mathcal{P})$.
Considering only minimal markings we also obtain that 
$\mathcal{M}_{\text{f}}^{1, \text{min}}(\mathcal{P}) = 
\mathcal{M}_{\text{loc}}^{1, \text{min}}(\mathcal{P})$.
\end{proof}

Markings from $\mathcal{M}_{\text{loc}}^{1}(\mathcal{P})$ can be described in the
cohomological way as follows. The chain group $C_{i}(\mathcal{P}, \mathbb{Z}_{2})$ is
the vector space over $\mathbb{Z}_{2} := \mathbb{Z}/2\mathbb{Z}$ generated by
$i$-dimensional simplices of $\mathcal{P}$. As usually the $i^{\text{th}}$ cochain group
is $C^{i}(\mathcal{P}, \mathbb{Z}_{2}) := C_{i}(\mathcal{P}, \mathbb{Z}_{2})^{\vee}$. Let
$\partial_{i}: C_{i}(\mathcal{P}, \mathbb{Z}_{2}) \longrightarrow 
C_{i-1}(\mathcal{P}, \mathbb{Z}_{2})$ be the boundary operator defined as follows: for any
$i$-dimensional simplex $\sigma$ we have $\partial_{i}(\sigma)$ equal to the sum over
$\mathbb{Z}_{2}$ of all elements of $C_{i-1}(\mathcal{P}, \mathbb{Z}_{2})$ corresponding
to facets of $\sigma$. The coboundary operator $\partial_{i}: 
C_{i}(\mathcal{P}, \mathbb{Z}_{2}) \longrightarrow C_{i+1}(\mathcal{P}, \mathbb{Z}_{2})$
is defined such that we have the identity $\partial^{i}(\delta)(\sigma) = 
\delta(\partial_{i}(\sigma))$ for any cochain $\delta \in 
C^{i}(\mathcal{P}, \mathbb{Z}_{2})$ and $(i+1)$-dimensional simplex $\sigma$. 
For any cocycle $\delta \in \text{ker}(\partial^{1})$ denote its support by 
$\text{supp}(\delta) \subset \mathcal{P}(1)$, i. e. the set of edges for which $\delta$ is 
not zero. Note that the group $\text{ker}(\partial^{1})$ has the partial order: 
$\delta_{1} \leq \delta_{2}$ if and only if 
$\text{supp}(\delta_{1}) \subseteq \text{supp}(\delta_{2})$.

\begin{Teo}\label{CocycleMarkings}
	\textit{There exists the canonical map 
		$f: \mathcal{M}_{\text{loc}}^{1}(\mathcal{P}) \DistTo 
			\text{ker}(\partial^{1})$ which is an isomorphism. 
		Moreover, a marking $M \in \mathcal{M}^{1}_{\text{loc}}(\mathcal{P})$ is minimal
		if and only if the 1-cocycle $f(M) \in \text{ker}(\partial^{1})$ is minimal.
	}
\end{Teo}
\begin{proof}
	The map $f$ is given by the construction of the vector 
	$f(M) \in P(\mathcal{P})^{\vee}_{M}$ for any 
	$M \in \mathcal{M}_{\text{loc}}^{1}(\mathcal{P})$
	from Lemma \ref{OneAdmissibleMarkings} considered over $\mathbb{Z}_{2}$. To show 
	that we obtain 1-cocycle consider any triangle $\Delta \in \mathcal{P}(2)$ and check 
	that $f(M)(\partial_{2}(\Delta)) = 0$. Indeed, since $M$ has no more than one marked 
	corner in each triangle we have two situations: the triangle has no marked corners 
	at all, or it has exactly one marked corner. Also note that $M$ is locally feasible, 
	so if there is no marked corner along the edge in some triangle then there is no such 
	marked corner in any triangle having this edge as its face. It implies that in 
	the first situation the cochain $f(M)$ vanishes over all edges of the triangle 
	$\Delta$, so we have $f(M)(\partial_{2}(\Delta)) = 0$ in this case. In the second 
	situation the cochain $f(M)$ vanishes over the edge opposite to the marked corner and 
	it is equal to $1$ over other two edges, so we have again 
	$f(M)(\partial_{2}(\Delta)) = 0$. Therefore, the cochain $f(M)$ is indeed a 1-cocycle, 
	$f(M) \in \text{ker}(\partial^{1})$, so this construction gives us the map 
	$f:\mathcal{M}_{\text{loc}}^{1}(\mathcal{P})\longrightarrow \text{ker}(\partial^{1})$.

	To show that the map $f$ is bijective we construct the inverse map 
	$f^{-1}:\text{ker}(\partial^{1}) \longrightarrow 
		\mathcal{M}_{\text{loc}}^{1}(\mathcal{P})$.
	Let $\delta \in \textnormal{ker}(\partial^{1})$ be a 1-cocycle over $\mathcal{P}$.
	We construct the associated marking $f^{-1}(\delta)$ in the following way.
	Consider any triangle $\Delta \in \mathcal{P}(2)$ and denote its edges by $e_{1}$, 
	$e_{2}$ and $e_{3}$. Since $\delta$ is a cocycle we have the equality
	\begin{equation}\label{CocycleEquation}
		\partial^{1}(\delta)(\Delta) = \delta \big(\partial_{2}(\Delta)\big) = 
			\delta (e_{1} + e_{2} + e_{3}) = 
			\delta(e_{1}) + \delta(e_{2}) + \delta(e_{3}) = 0.
	\end{equation}
	There can be two situations over $\mathbb{Z}_{2}$: $\delta$ either vanishes over all 
	three edges $e_{1}$, $e_{2}$, $e_{3}$, or it does not vanish over exactly two of these 
	three edges, say, over $e_1$ and $e_2$. For the first case we have that no one corner 
	of $\Delta$ is marked in $f^{-1}(\delta)$. For the second case we have that 
	the triangle $\Delta$ has exactly one corner which is marked in $f^{-1}(\delta)$, 
	namely, $(\Delta, [e_1, e_2])$. Therefore, we obtain the marking $f^{-1}(\delta)$ 
	which has no more than one marked corner in each triangle, so 
	$f^{-1}(\delta) \in \mathcal{M}^{1}(\mathcal{P})$. Now show that this marking is 
	locally feasible. Consider any two triangles $\Delta_1$, $\Delta_2$ with a common 
	edge $e$. Assume that the edges of $\Delta_{1}$ are $e_{1}$, $e_{2}$ and $e_{3}$, 
	while the edges of $\Delta_{2}$ are $e_{3}$, $e_{4}$ and $e$. Then we have 
	the following two equalities over $\mathbb{Z}_2$:
	\begin{equation}\label{CocycleAdmissibilityEqs}
		\delta\big(\partial_{2}(\Delta_1)\big) = 
			\delta(e_{1}) + \delta(e_{2}) + \delta(e) = 0, \ \ \
		\delta\big(\partial_{2}(\Delta_2)\big) = 
			\delta(e_{3}) + \delta(e_{4}) + \delta(e) = 0. 
	\end{equation}
	Suppose that the edge $e$ is marked in some of these two triangles, say $\Delta_{1}$.
	By construction it follows that $\delta(e) \neq 0$ and exactly one of two values 
	$\delta(e_{1})$ or $\delta(e_{2})$ also does not vanish. From the fact that
	$\delta(e) \neq 0$ and the second equality of (\ref{CocycleAdmissibilityEqs}) we 
	obtain that exactly one of the values $\delta(e_{3})$ or $\delta(e_{4})$ does not 
	vanish as well. It means that the triangle $\Delta_{2}$ also contains a marked corner 
	along the edge $e$, so $e$ is marked in $\Delta_{2}$. Therefore, we see that 
	the marking $f^{-1}(\delta)$ is locally feasible, i. e. $f^{-1}(\delta) \in
	\mathcal{M}_{\text{loc}}^{1}(\mathcal{P})$. From the construction it also can be seen
	that $f \circ f^{-1}$ and $f^{-1} \circ f$ are identity maps.

	Next, proceed to the proof of the second statement of the theorem. Assume that 
	the marking $M$ is not minimal in $\mathcal{M}^{1}_{\text{loc}}(\mathcal{P})$. 
	It means that there exists some marking 
	$M' \in \mathcal{M}^{1}_{\text{loc}}(\mathcal{P})$ such that $M' \subset M$. Consider 
	the corresponding 1-cocycle $\delta' = f(M') \in \text{ker}(\partial^{1})$. Take any 
	edge $e$ belonging to the support of $\delta'$, i. e. $e\in\text{supp}(\delta')$.
	By the construction of $M'$ any triangle having $e$ as its facet is marked along this 
	edge. Since $M' \subset M$ the same can be said about marking $M$, so 
	$e \in \text{supp}(f(\delta))$. Therefore, we obtain that 
	$\text{supp}(f(M')) \subset \text{supp}(f(M))$ which means that $f(M)$ is not minimal.

	Now suppose that $\delta \in \text{ker}(\partial^{1})$ is not minimal. In particular, 
	it means that there is the decomposition $\delta = \delta_1 + \delta_2$ such that 
	$\delta_1, \delta_2 \in \text{ker}(\partial^{1})$ and 
	$\text{supp}(\delta_{1}) \cap \text{supp}(\delta_{2}) = \emptyset$. 
	Consider the induced locally feasible markings 
	$f^{-1}(\delta_{1}), f^{-1}(\delta_{2})\in \mathcal{M}^{1}_{\text{loc}}(\mathcal{P})$. 
	Assume that $f^{-1}(\delta_{1})$ has a marked corner in some triangle $\Delta$ between 
	edges $e_{1}$ and $e_{2}$, i. e. $(\Delta, [e_1, e_2]) \in f^{-1}(\delta_{1})$. 
	It implies that $\delta_{1}(e_1) = \delta_{1}(e_2) = 1$, so 
	$\{e_1, e_2\} \subset \text{supp}(\delta_{1})$.
	Since $\text{supp}(\delta) = \text{supp}(\delta_{1}) \sqcup \text{supp}(\delta_{2})$ 
	we have that $\{ e_{1}, e_{2} \} \subset \text{supp}(\delta)$, but it necessary leads 
	to that the corner $(\Delta, [e_1, e_2])$ is also marked in $f^{-1}(\delta)$. 
	Therefore, we have the inclusion $f^{-1}(\delta_{1}) \subset f^{-1}(\delta)$, so 
	$f^{-1}(\delta)$ is not minimal in  $\mathcal{M}^{1}_{\text{loc}}(\mathcal{P})$. 
	Actually, using the similar ideas one can show that 
	$f^{-1}(\delta) = f^{-1}(\delta_{1}) \sqcup f^{-1}(\delta_{2})$.
\end{proof}

Now suppose that the first cohomology vanishes, $\text{H}^{1}(\mathcal{P}, \mathbb{Z}_2) 
\simeq 0$, then we have the canonical isomorphism $\text{ker}~\partial^{1} \simeq 
\text{im}~\partial^{0}$. It means that for each 1-cocycle $\delta$ over $\mathbb{Z}_2$ 
there exists a subset of vertices $S \subset \mathcal{P}(0)$ considered as a 0-cocycle 
such that $\delta = \partial^{0}(S)$. On the other hand, we have the relation 
$\partial^{0}(S) + \partial^{0}(\overline{S}) = \partial^{0}(\mathcal{P}^{(0)}) = 0$ 
where $\overline{S}$ is the complement of the subset $S$ in $\mathcal{P}(0)$, 
so we also have that $\delta = \partial^{0}(S) = \partial^{0}(\overline{S})$. 
It follows that the structure of $\mathcal{M}_{\text{loc}}^{1}(\mathcal{P})$ for this case 
can be described in purely graph theoretical notions. To do this recall some definitions. 
Let $G$ be a graph and assume that we have some partition $\{ S, \overline{S}\}$ of 
the set of vertices $V(G)$ into two disjoint subsets. 
The subset $H\big( \{ S, \overline{S}\} \big) \subset E(G)$ consisted of the edges which 
have exactly one end-point in the set $S$ and other end-point in the set $\overline{S}$ 
is called a \textit{cutset} induced by the partition $\{ S, \overline{S}\}$. All possible 
cutsets of the graph $G$ form a vector space over $\mathbb{Z}_{2}$ which is called 
a \textit{cut space}. The relation of inclusions of sets induces the partial order on 
the cut space. Now, using this terminology, we can formulate the following:

\begin{Cor}\label{ParticularCaseOneMarkings}
	\textit{If $\text{\emph H}^{1}(\mathcal{P}, \mathbb{Z}_2) \simeq 0$ then 
		$\mathcal{M}^{1}_{\text{loc}}(\mathcal{P})$ is isomorphic to the cut space of 
		the graph $\mathcal{P}^{(1)}$. Moreover, minimal markings from
		$\mathcal{M}^{1}_{\text{loc}}(\mathcal{P})$ correspond to minimal cutsets.}
\end{Cor}
\begin{proof}
	For any marking $M \in \mathcal{M}_{\text{loc}}^{1}(\mathcal{P})$ we have that
	$f(M)=\partial^{0}(S) = \partial^{0}(\overline{S}) \in \text{ker}~\partial^{1} \simeq 
	\text{im}~\partial^{0}$ for some $S \subset \mathcal{P}(0)$, so we obtain the
	partition $\{ S, \overline{S} \}$ of $\mathcal{P}(0)$. Moreover, $\text{supp}(f)$ 
	consists from those edges which connect a vertex belonging to $S$ with another vertex 
	which does not belong to $S$. 
	Therefore, $\text{supp}(f) = H\big( \{ S, \overline{S}\} \big)$ is a cutset of 
	the graph $\mathcal{P}^{(1)}$ induced by the partition $\{S, \overline{S}\}$.

	Now if we have a cutset $H(\{S, \overline{S}\})$ for some partition 
	$\{S, \overline{S}\}$ then we can consider a unique cochain $\delta$ which has this 
	cutset as its support. From the graph theory we know that any cutset has 
	an intersection of even size with any cycle. It implies that $\delta$ is 1-cocycle 
	over $\mathbb{Z}_{2}$ such that $\delta =\partial^{0}(S)=\partial^{0}(\overline{S})$.
	So we obtain the locally feasible marking 
	$f^{-1}(\delta) \in \mathcal{M}_{\text{loc}}^{1}(\mathcal{P})$.
	The consistency of the notion of minimality follows immediatelly from 
	Theorem \ref{CocycleMarkings}.  
\end{proof}

Remind again that the facets of the cosmological polytope $P(G)$ are in one-to-one
correspondence with connected subgraphs of the graph $G$. Let us show that the obtained
result for ternary polytope is consisted with this statement for cosmological polytopes.
Firstly, note that minimality of a cutset $H\big( \{ S, \overline{S} \}\}\big)$ of 
a graph $G$ is equivalent to the following requirement. For any subset $U \subset V(G)$
denote the induced subgraph of $G$ which consists of the vertices from $U$ and all edges 
of $G$ which connect the vertices from $U$ by $G_{U} \subset G$. Using this notation 
one can state that the cutset $H\big( \{ S, \overline{S} \}\}\big)$ is minimal 
if and only if the induced subgraphs $G_{S}$ and $G_{\overline{S}}$ are both connected. 

As it was shown in the previous section the ternary relation $\mathcal{T}(G)$ generated 
by a graph $G$ coincides with ternary relation $\mathcal{T}\big( CG \big)$ generated by 
the cone over the graph $G$. By construction of a cone over a graph we have that for any 
partition $\{S, \overline{S}\}$ of the set of vertices $CG(0)$ of the graph $CG^{(1)}$, 
i. e. $CG(0) = S \sqcup \overline{S}$, only one subset, $S$ or $\overline{S}$, does not 
contain the apex of $CG$. It means that any partition of $CG(0)$ uniquely determines 
subset of vertices of the original graph $G$. Moreover, the condition of minimality of 
the considered cutset is equivalent to the connectivity condition of the corresponding 
induced subgraph of $G$. So minimal cutsets of the graph $CG^{(1)}$ are in one-to-one 
correspondence with connected induced subgraphs of $G$. Therefore, we have the corollary:

\begin{Cor}\label{GraphMinimalMarkings}
	For any graph $G$ we have
	$\mathcal{M}_{\text{loc}}^{1, \text{min}}(CG)$ $\simeq$
	$\{ S \subset V(G) \ | \ G_{S} \text{ is connected} \}$.
\end{Cor}
\noindent
Note that the markings from $\mathcal{M}_{\text{loc}}^{1, \text{min}}(CG)$ do not exhaust 
in general all possible minimal (locally) feasible markings, i. e.
$\mathcal{M}_{\text{loc}}^{1, \text{min}}(CG) \subsetneq
\mathcal{M}_{\text{loc}}^{\text{min}}(CG) = \mathcal{M}_{\text{f}}^{\text{min}}(CG)$.
Indeed, $CG$ can have minimal feasible markings for which some triangles have 2 marked
corners. One can show that such markings are associated with connected subgraphs of $G$ 
which are not induced by any subset $S \subset V(G)$. 

Now let us look to the case of root systems. We can use the characterization of minimality 
of cutsets in terms of induced subgraphs to prove the following statement. Since
$\mathcal{K}_{n}$ is the face poset of the standard $(n-1)$-dimensional simplex we have
that $\text{H}^{1}(\mathcal{K}_{n}, \mathbb{Z}_{2}) \simeq 0$. Moreover, since
$\mathcal{K}_{n}^{(1)}$ coincides with the complete graph $K_{n}$ we have that for any
vertex subset $S \subset \mathcal{K}_{n}(0)$ the corresponding induced subgraph is always
connected. Therefore, applying Corollary \ref{ParticularCaseOneMarkings} to this case 
and taking into account that $\mathcal{T}(A_{n-1}) \simeq \mathcal{T}(\mathcal{K}_{n})$ 
we obtain the following:

\begin{Cor}\label{AnMinimalOneMarkings}
	\textit{For any $n > 0$ the Graev polytope $P(A_{n-1})$ has the family of facets which
		are in one-to-one correspondence with subsets of $\mathcal{K}_{n}(0) =
		\{ v_1, .., v_{n} \}$ considered up to taking compliment, 
		i. e. $S \sim \mathcal{K}_{n}(0) \setminus S$.
	}
\end{Cor}

\noindent
Note that the first cohomology over $\mathbb{Z}_{2}$ does not vanish for the simplicial
posets $\overline{\mathcal{K}}_{n}$ and $\big( C\overline{\mathcal{K}}_{n} \big)^{(2)}$, 
so we cannot apply Corollary \ref{ParticularCaseOneMarkings} to them and need to find out
auxiliary cocycles explicitly. Corollary \ref{AnMinimalOneMarkings} describes all
facets of $P(A_{n})$ for $n = 2$ and $n = 3$. However, for $n \geq 4$ it is not true.
Moreover, there are facets of $P(A_{n}), n \geq 4$ whose corresponding markings do not 
even belong to $\mathcal{M}^{\text{min}}_{\text{loc}}(\mathcal{K}_{n+1})$. 
The next section is devoted to the study of some class of such facets.

\section{Extreme metrics}\label{Extremal metrics}

\subsection{Metrics on 2-dimensional simplicial posets}

Note that any edge of $\mathcal{K}_{n}$ is uniquely determined by two end-vertices. 
It means that a covector $\mu \in P(\mathcal{K}_{n})^{\vee}$ considered as a function over
the set of edges uniquely determines the symmetric function 
$d_{\mu}:\mathcal{K}_{n}(0) \times \mathcal{K}_{n}(0) \longrightarrow \mathbb{R}$
such that 
$d_{\mu}(v_{i}, v_{j}) = d_{\mu}(v_{j}, v_{i}):=\mu(e_{ij})$.
Moreover, the linear inequalities which define the cone (\ref{ConeDefinition}) can be 
rewritten in the following form:
\begin{equation}
	\begin{cases}
		d_{\mu}(v_{i}, v_{j}) + d_{\mu}(v_{j}, v_{k}) \geq d_{\mu}(v_{i}, v_{k}), \\
		d_{\mu}(v_{i}, v_{k}) + d_{\mu}(v_{k}, v_{j}) \geq d_{\mu}(v_{i}, v_{j}), \\
		d_{\mu}(v_{j}, v_{i}) + d_{\mu}(v_{i}, v_{k}) \geq d_{\mu}(v_{j}, v_{k}).
	\end{cases}
\end{equation}
Symmetric two-point function satisfying these inequalities is called 
\textit{a metric}\footnote{We do not require the condition 
$d(x, y) = 0 \Leftrightarrow x = y$ in the definition of a metric.} on the finite set
consisted of $n$ points. The set of all metrics on $n$ points is a convex polyhedral cone
in $\mathbb{R}^{{n \choose 2}}$ which is called \textit{the metric cone} and denoted by
$M_{n}$. In opposite, any metric $d$ over the finite set $\{v_1, .., v_n\}$ uniquely 
determines the covector $\mu_{d} \in P(\mathcal{K}_{n})^{\vee}$ defined as 
$\mu_{d}(e_{ij}):=d(v_{i},v_{j})$. Therefore, we have the following statement:
\begin{Lem}\label{MetricConeIsomorphism}
	\textit{The dual cone $P(\mathcal{K}_{n})^{\vee} = P(A_{n-1})^{\vee}$ is isomorphic to
	the metric cone $M_{n}$.
	}
\end{Lem}
\noindent
From this perspective we can think that the cone $P(\mathcal{P})^{\vee}$ for any
simplicial poset $\mathcal{P}$ is some generalization of the metric cone. More precisely,
consider the following definition.

\begin{Definition}\label{PosetMetric}
Let $\mathcal{P}$ be a 2-dimensional simplicial poset. The non-negative function 
$d: \mathcal{P}(1) \longrightarrow \mathbb{R}_{\geq 0}$ is called a metric on 
$\mathcal{P}$ if it satisfies the following system of inequalities for any triangle
$\Delta \in \mathcal{P}(2)$ with facets $e_1$, $e_2$, $e_3$:
\begin{equation}\label{TriangleInequalitiesOnComplexes}
	\begin{cases}
		d(e_1) + d(e_2) \geq d(e_3), \\
		d(e_2) + d(e_3) \geq d(e_1), \\
		d(e_1) + d(e_3) \geq d(e_2).
	\end{cases}
\end{equation}
\end{Definition}
\noindent
There are many generalizations of the notion of a metric (see \cite{DL}), but it seems
that Definition \ref{PosetMetric} did not appear in the mathematical literature 
previously. It is evident that the set of all such metrics on $\mathcal{P}$ coincides with 
the cone $P(\mathcal{P})^{\vee}$. Further in this section we will use this terminology 
applied to the cone $P(\mathcal{P})^{\vee}$.

Remind that the facets of the ternary polytope $P(\mathcal{P})$ are in one-to-one
correspondence with extreme rays of its dual cone $P(\mathcal{P})^{\vee}$. The non-zero
covectors of these extreme rays can be seen as \textit{extreme metrics} on $\mathcal{P}$,
i. e. those metrics which cannot be obtained as the sum of two distinct metrics not
proportional to the original one. In particular, we have the following statement:

\begin{Cor}\label{ExtremalMetricsFacetsCorrespondence}
	\textit{Facets of the Graev polytope $P(A_{n-1})$ are in one-to-one correspondence
	with extreme metrics on $n$ points (up to homothety).
	}
\end{Cor}

\noindent
A lot of mathematical papers is devoted to the problem of description of extreme metrics 
on $n$ points (see \cite{BD, Av, Av2, Gr}). For small values of $n$ the extreme metrics 
are fully classified. However, the complete classification of extreme metrics for 
arbitrary $n$ is still not done. Several types of extreme metrics are known. The simplest
type is given by so-called cut-metrics on $n$ points which are always extreme. Their
construction is following. Consider the set of $n$ points as the set of vertices of 
the complete graph $K_{n}$, so any metric on these $n$ points is a non-negative function 
over the set $E(K_{n})$ of edges of $K_{n}$ satisfying the triangle inequalities. For any 
partition $V(K_{n}) = S \sqcup \overline{S}$ we can consider the corresponding cutset 
$H(\{ S, \overline{S}) \subset E(K_{n})$ as it was discussed in the previous section. 
This cut-set defines the metric $d_{\{ S, \overline{S} \}}$ which is equal to 1 on edges 
connecting vertices from different subsets $S$ and $\overline{S}$, while for all other 
edges it is equal to zero 
(see in Fig. \ref{Cut_metrics_on_three_points} all cut-metrics on 3 points). 
It is easy to see that all triangle inequalities are valid, so $d_{\{ S, \overline{S} \}}$
is indeed a metric. Moreover, one can prove that it is extreme, so it generates 
an extreme ray of $M_{n}$. Note that this is equivalent to the statement of 
Corollary \ref{AnMinimalOneMarkings}, so 
the result of Corollary \ref{AnMinimalOneMarkings} is actually not new. On the other hand, 
Corollary \ref{AnMinimalOneMarkings} is the consequence of Theorem \ref{CocycleMarkings} 
for the partial case $\mathcal{P}=\mathcal{K}_{n}$. Therefore, it is natural to consider 
metrics on $\mathcal{P}$ corresponding to locally feasible markings from
$\mathcal{M}_{\text{loc}}^{1}(\mathcal{P})$ as generalization of cut-metrics on $n$ 
points.

\begin{figure}[t]
\includegraphics[width=.9\textwidth]{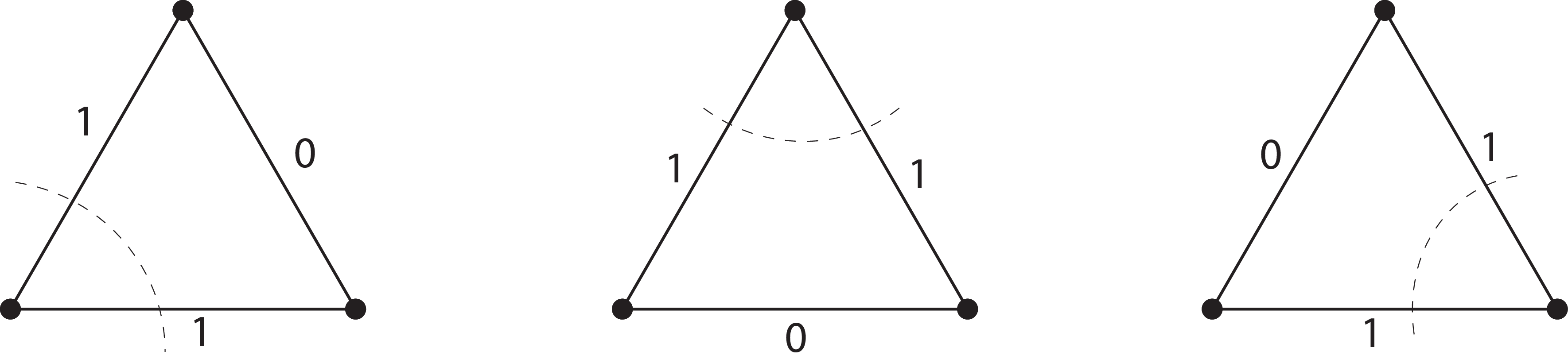}
\caption{Cut-metrics on three points.}\label{Cut_metrics_on_three_points}
\end{figure}

More complicated type of extreme metrics is given by certain graph metrics induced by
subgraphs of $K_{n}$. More precisely, assume that we have some connected subgraph 
$G \subset K_{n}$ containing all vertices of $K_{n}$. This subgraph induces 
\textit{the graph metric} $d_{G}$ over $n$ points. By definition the metric $d_{G}$ 
between any two vertices is equal to the length of the shortest path in $G$ connecting 
these vertices. However, the metric constructed in this way is not necessarily extreme. 
Avis provided in \cite{Av} the sufficient condition for a subgraph $G$ which guarantees that 
the corresponding graph metric is extreme. The first occurence of such extreme metrics 
happens for $5$ points, namely, the graph metric induced by the bipartite graph 
$K_{3,2} \subset K_{5}$ is extreme (see Fig. \ref{K_3_2_metric}). This condition can be 
formulated in terms of specific coloring properties of a graph.
We will extend Avis' approach to the case of metrics on simplicial posets.

\begin{figure}[t]
\includegraphics[width=.5\textwidth]{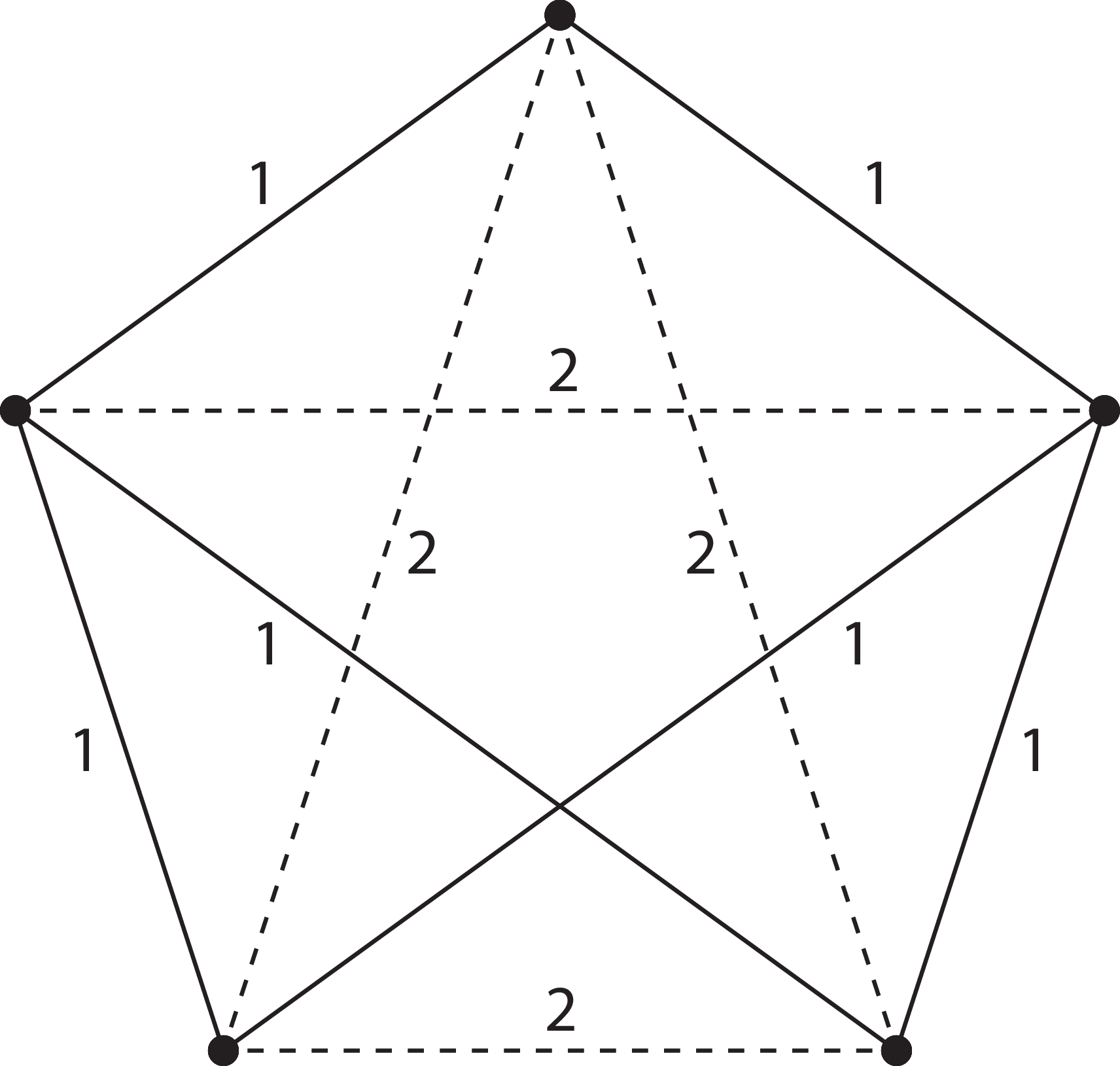}
\caption{The extreme metric on 5 points associated with the subgraph $K_{3,2}$ of
	$K_{5}$. The subgraph $K_{3,2}$ is highlighted by solid line.}\label{K_3_2_metric}
\end{figure}

\subsection{Generalization of the Avis' result}

Consider arbitrary pure 2-dimensional simplicial poset  $\mathcal{P}$. A pair consisted of 
the sequence of edges $e_{1}.. e_{n}$ and the sequence of vertices $v_{1}.. v_{n+1}$ such 
that $e_{i}$ connects $v_{i}$ with $v_{i+1}$, $1 \leq i \leq n$, is called \textit{a walk} 
on $\mathcal{P}$. Note that edges and vertices in a walk can repeat. For simplicity, 
we will often denote a walk just by the sequence of edges $e_{1}.. e_{n}$ not explicitly 
indicating the sequence of vertices. \textit{The length of a walk} is the number of edges 
used in it. \textit{A path} is a walk which has no repetitions in edges. A path is called 
\textit{simple} if it has no repetitions in vertices. \textit{A cycle} is a path with 
the first and the last vertices coincided, i. e. $v_1 = v_{n+1}$. For any two walks 
$p=e_1 .. e_n$ and $p'=e_1'..e_m'$ such that the last vertex of $p$ coincides with 
the first vertex of $p'$ we can consider their \textit{concatenation} which is the walk 
giving by the sequence of edges $e_1..e_n e_1' .. e_m'$ and denoted by $p \cup p'$. 

Suppose now that a walk $p$ has two subsequent edges $e_i$, $e_{i+1}$ for some 
$1 \leq i \leq n$ such that there exists an edge $\tilde{e} \in \mathcal{P}(1)$ with 
the property that the edges $e_{i}$, $e_{i+1}$ and $\tilde{e}$ together form 
a triangle $\Delta \in \mathcal{P}(2)$. The modified walk 
$\tilde{p} = e_{1} ... e_{i-1} \tilde{e} e_{i+2} ... e_{n}$ is called 
an \textit{elementary contraction} of the walk $p$ along $\Delta$. Also we will say that 
a walk $p$ is a \textit{contraction} of another walk $p'$ along $\mathcal{P}$, or that 
$p'$ \textit{can be contracted} to $p$ if the walk $p$ can be obtained from $p'$ by 
sequence of elementary contractions along triangles of $\mathcal{P}$. 
From the construction it follows that the set of vertices of any contraction of a walk is 
always the subset of the set of vertices of the walk itself. 

For any edge $e$ of $\mathcal{P}$ we will say that a walk $p$ is 
\textit{a bypassing walk} of the edge $e$ if $p$ connects the vertices of $e$. Moreover, 
if such a walk can be contracted to $e$ we call it \textit{a contractable bypassing walk} 
of the edge $e$. Further, we will say that a subgraph $G \subset \mathcal{P}^{(1)}$ is 
\textit{a bypassing subgraph in} $\mathcal{P}$ if it contains all vertices of 
$\mathcal{P}$ and all edges of $\mathcal{P}$ have contractable bypassing walks completely 
contained in $G$.

For any subgraph $G \subset \mathcal{P}^{(1)}$ and edge $e \in \mathcal{P}(1)$ denote by 
$B_{G}^{\mathcal{P}}(e)$ the set of shortest walks among all walks in $G$
contractable to the edge $e$ along $\mathcal{P}$. It is evident that if
$B_{G}^{\mathcal{P}}(e) \neq \emptyset$ then all walks in $B_{G}^{\mathcal{P}}(e)$ have
the same length which we denote by $d_{G}^{\mathcal{P}}(e) \in \mathbb{Z}_{> 0}$. If $G$
is a bypassing graph in $\mathcal{P}$ then we have $B_{G}^{\mathcal{P}}(e) \neq \emptyset$
for any edge $e \in \mathcal{P}(1)$, so the number $d_{G}^{\mathcal{P}}(e)$ is
well-defined for any edge of $\mathcal{P}$. Since $B_{G}^{\mathcal{P}}(e) = \{ e \}$ for
$e \in E(G)$ we obtain that $d_{G}^{\mathcal{P}}$ is equal to 1 on the edges of
the subgraph $G$, while for all other edges of $\mathcal{P}$ it is strictly bigger than 1.   

\begin{Lem}\label{ShortestDistanceMetric}
	For any bypassing subgraph $G \subset \mathcal{P}^{(1)}$ the function 
	$d_{G}^{\mathcal{P}}$ defined over the edges of $\mathcal{P}$ is a metric on 
	$\mathcal{P}$. We call $d_{G}^{\mathcal{P}}$ the graph metric on $\mathcal{P}$ induced
	by $G \subset \mathcal{P}^{(1)}$.
\end{Lem}
\begin{proof}
	Consider any triangle $\Delta \in \mathcal{P}(2)$ and denote its edges by 
	$e_1$, $e_2$, $e_3 \in \mathcal{P}(1)$. For each $i \in \{1, 2, 3\}$ fix the shortest 
	walk $p_{i} \in B_{G}^{\mathcal{P}}(e_{i})$. According to the definition we have that 
	$d_{G}^{\mathcal{P}}(e_i)$ is equal to the length of the walk $p_{i}$. Now suppose 
	that some of triangle inequalities associated with $\Delta \in \mathcal{P}(2)$ is 
	not satisfied. Without loss of generality we can assume that 
	$d_{G}^{\mathcal{P}}(e_1) > d_{G}^{\mathcal{P}}(e_2) + d_{G}^{\mathcal{P}}(e_3)$. 
	Consider the concatenated walk $p_{23}=p_{2} \cup p_{3}$. This path connects 
	the vertices of the edge $e_{1}$ and, moreover, it can be contracted to $e_{1}$. 
	Indeed, $p_{23}$ can be contracted to the walk $e_{2} \cup e_{3}$ which in its turn 
	can be contracted to the edge $e_3$ since $\mathcal{P}$ has the triangle $\Delta$.
	On the other hand, the value of $d_{G}^{\mathcal{P}}(e_1)$ is strictly more than 
	the length of the walk $p_{23}$ which contradicts with 
	the construction of $d_{G}^{\mathcal{P}}$.
\end{proof}
\noindent
See in Fig. \ref{metric_on_tiling} the example of the metric of the form
$d_{G}^{\mathcal{P}}$ where $\mathcal{P}$ is the cone over 3-cycle (note that
$\mathcal{P}$ contains 3 triangles, not $4$ as $\mathcal{K}_{4}$ has). 

\begin{figure}[t]
\includegraphics[width=.5\textwidth]{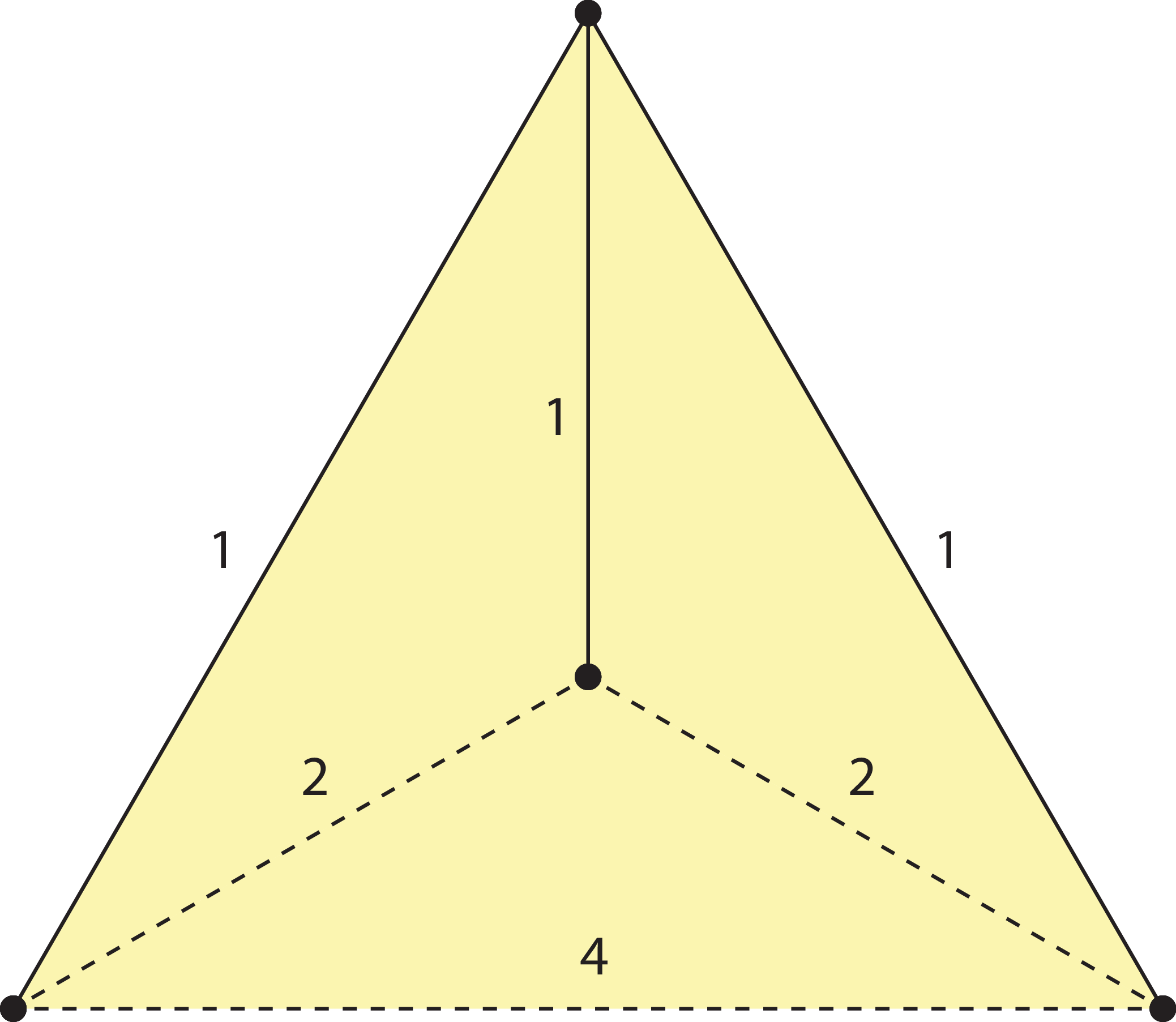}
	\caption{The graph metric on the cone over 3-cycle. This poset contains only 
	3 triangles. The subgraph of the 1-skeleton which induces the metric is highlighted by
	solid line. Note that the bottom edge has the contractable bypassing walk with
	repetitions. Also note that this metric is not extreme.}\label{metric_on_tiling}
\end{figure}

In the remaining of this section we will omit the upper index in $B_{G}^{\mathcal{P}}$ and
$d_{G}^{\mathcal{P}}$ if it is clear from the context what poset is considered. 
Consider the question when $d_{G}$ is an extreme metric on $\mathcal{P}$. We will present 
a sufficient condition for the graph $G \subset \mathcal{P}^{(1)}$ which generalizes 
the Avis' result. One of the main observations needed to formulate this condition 
is the following lemma: 
\begin{Lem}\label{TightConstraints}
	\textit{Let $d$ be any metric on $\mathcal{P}$. Assume that we have decomposition 
		$d = d' + d''$ where $d'$, $d''$ are some metrics on $\mathcal{P}$. If the metric
		$d$ satisfies the equality $d(e_1) + d(e_2) = d(e_3)$ for some edges $e_1$, $e_2$,
		$e_3$ which are facets of some triangle $\Delta \in \mathcal{P}(2)$ then this
		equality is true for $d'$ and $d''$ as well.}
\end{Lem}
\begin{proof}
	Suppose that one of the metrics $d'$ or $d''$, say $d'$, does not satisfy the equality
	which means that we have strong inequality $d'(e_1) + d'(e_2) > d'(e_3)$. Since in any 
	case the metric $d''$ satisfies the inequality $d''(e_1) + d''(e_2) \geq d''(e_3)$ 
	we obtain the strong inequality $d(e_1) + d(e_2) > d(e_3)$ which contradicts 
	the assumptions. 
\end{proof}

\noindent
Also we will need the following technical lemma:
\begin{Lem}\label{TechnicalLemma}
	\textit{Let $p$ be a walk with the sequence of edges $e_{1}...e_{n}$ and the sequence
		of vertices $v_{1}...v_{n+1}$, and $p'$ be its contraction along $\mathcal{P}$. 
		Assume that $\gamma$ is arbitrary edge of the walk $p'$. If $\gamma$ connects 
		the vertices $v_{i}$, $v_{j}$, $1 \leq i < j \leq n$, then $\gamma$ is contraction 
		of the walk $e_{i}...e_{j-1}$. 
	}
\end{Lem}
\begin{proof}
	By definition, there exists the sequence of walks $p_{1}$, ..., $p_{m}$ such that
	$p_{1}=p$, $p_{m}=p'$ and for any $1 \leq i \leq m-1$ the walk $p_{i+1}$ is 
	the elementary contraction of $p_{i}$ along some triangle 
	$\Delta_{i} \in \mathcal{P}(2)$. To prove the statement we use the induction by 
	the number of elementary contractions $m$ needed to obtain $p'$. For $m = 1$ 
	the statement is trivial because $p = p'$ in this case. Now suppose that 
	the statement is true for any $m \leq m_{0}$ and consider the case $m=m_{0}+1$. 
	By the construction, the walk $p' = p_{m_{0}+1}$ is the elementary contraction of 
	$p_{m_{0}}$ along some triangle $\Delta_{m_{0}}$. If $\gamma$ is not a facet of 
	$\Delta_{m_{0}}$ then $\gamma$ still belongs to the walk $p_{m_{0}}$ for which 
	the statement of the lemma holds by the assumption of induction. On the other hand, 
	if $\gamma$ is an edge of the triangle $\Delta_{m_{0}}$ then there should exist two 
	subsequent edges $\gamma_{1}$ and $\gamma_{2}$ of the walk $p_{m_{0}}$ such that 
	$\gamma_{1}$ connects $v_{i}$ with $v_{k}$, $\gamma_{2}$ connects $v_{k}$ with $v_{j}$ 
	for some $k$ satisfying the inequality $1 \leq i < k < j \leq n+1$, and 
	three edges $\gamma$, $\gamma_{1}$, $\gamma_{2}$ are facets of the triangle 
	$\Delta_{m_{0}}$. By assumption of the induction, $\gamma_{1}$ is the contraction of 
	$e_{i}...e_{k-1}$, while $\gamma_{2}$ is the contraction of $e_{k}...e_{j-1}$. 
	Therefore, $\gamma$ is the contraction of concatenation 
	$e_{i}...e_{k-1} \cup e_{k}...e_{j-1}$ which is just the walk $e_{i}...e_{j-1}$.
\end{proof}
\noindent
Using this lemma one can prove the following:
\begin{Teo}\label{PathThm}
	\textit{Assume that there is a decomposition $d_{G} = d' + d''$ for some metrics $d'$,
		$d''$ on $\mathcal{P}$. Let $e$ be any edge of $\mathcal{P}$ such that $B_{G}(e)
		\neq \emptyset$. Then for any walk $p=e_1 ... e_n \in B_{G}(e)$ we have 
		the equality $d'(e_1) + .. + d'(e_{n}) =  d'(e)$.
	}
\end{Teo}
\begin{proof}
	Let $v_1 ... v_{n+1}$ be the sequence of vertices of the walk $p$. By the assumption,
	there exists the sequence of walks $p_1$, ..., $p_n$ such that $p_{1}=p$, $p_{n}=e$ 
	and for any $1 \leq i \leq n-1$ the walk $p_{i+1}$ is the elementary contraction of 
	$p_{i}$ along some triangle $\Delta_{i} \in \mathcal{P}(2)$. Fix arbitrary 
	$m \in \{1, ..., n\}$ and consider an edge $\gamma$ belonging to the walk $p_{m}$. 
	Suppose that $\gamma$ connects $v_{i}$ with $v_{j}$ for some $1 \leq i < j \leq n+1$. 
	From Lemma \ref{TechnicalLemma} it follows that $\gamma$ is a contraction of the walk 
	$e_{i}...e_{j-1}$. In particular, it means that $d_{G}(\gamma) \leq j - i$. 
	On the other hand, if $d_{G}(\gamma) > j - i$ then there should exists a walk 
	$p_{\gamma}$ in $G$ whose length is strictly smaller than the length of 
	$e_{i}...e_{j-1}$. However, in this case the concatenated walk 
	$e_{1}...e_{i-1} \cup p_{\gamma} \cup e_{j}...e_{n}$ can be contracted to $p_{m}$ and, 
	consequently, to the edge $e$. But it contradicts to that the walk $p$ is the shortest 
	contractable to $e$ walk in $G$. Therefore, we have that $d_{G}(\gamma) = j - i$.

	Now consider the sequence of triangles 
	$\Delta_{1}$, ..., $\Delta_{n-1} \in \mathcal{P}(2)$. 
	Fix arbitrary $m \in \{1, ..., n-1 \}$ and note that the triangle $\Delta_{m}$ 
	should contain three edges $\gamma_{1}$, $\gamma_{2}$, $\gamma_{3}$ such that 
	$\gamma_{1}$, $\gamma_{2}$ are subsequent edges of the walk $p_{m}$, while
	$\gamma_{3}$ belongs to $p_{m+1}$. The vertices of $\Delta_{m}$ are some $v_{i}$,
	$v_{j}$, $v_{k}$, $1 \leq i < j < k \leq n+1$, such that $\gamma_{1}$ connects $v_{i}$ 
	with $v_{j}$, $\gamma_{2}$ connects $v_{j}$ with $v_{k}$ and $\gamma_{3}$ connects 
	$v_{i}$ with $v_{k}$. As it was proved above we have that $d_{G}(\gamma_{1})=j-i$,
	$d_{G}(\gamma_{2})=k-j$ and $d_{G}(\gamma_{3})=k-i$ which implies the equality
	$d_{G}(\gamma_{1})+d_{G}(\gamma_{2})=d_{G}(\gamma_{3})$. 
	Applying Lemma (\ref{TightConstraints}) we obtain the similar equality for 
	the metric $d'$, i. e. $d'(\gamma_{1})+d'(\gamma_{2})=d'(\gamma_{3})$.

	Now denoting for each $m$ the edges of $\Delta_{m}$ chosen as above by 
	$\gamma_{m,1}$, $\gamma_{m,2}$, $\gamma_{m,3}$ we obtain the system of linear
	equalities $d'(\gamma_{m,1})+d'(\gamma_{m,2})=d'(\gamma_{m,3})$, $1 \leq m \leq n-1$. 
	Applying these equalities to the expression $d'(e_1)+...d'(e_{n})$ we reduce it 
	to $d'(e)$. Therefore, we have the equality $d'(e_1)+...+d'(e_n)=d'(e_n)$.
\end{proof}

Consider \textit{a shift operator} $s_{n}: \{1, ..., n\} \longrightarrow \{1, ..., n\}$
defined as follows: $s_{n}(i) = i+1$ for any $1 \leq i \leq n-1$, and $s_{n}(n)=1$. We
denote its composition with itself $m$ times by $s_{n}^{m}$, i. e. $s_{n}^{m} = s_{n}
\circ (m \text{ times}) \circ s_{n}$. Using this notation we can formulate the following
statement:

\begin{Cor}\label{CycleCor}
	\textit{Assume that there is a decomposition $d_{G} = d' + d''$ for some metrics $d'$,
		$d''$ on $\mathcal{P}$. Let $C$ be a cycle in $G$ with even number of edges whose 
		sequence of edges is $e_1... e_{2n}$ and sequence of vertices is $v_1... v_{2n}$. 
		Suppose that for any $1 \leq i \leq n$ there exists an edge $\widetilde{e}_{i} 
		\in \mathcal{P}(1) \setminus E(G)$ connecting vertices $v_{i}$ and 
		$v_{s_{2n}^{n}(i)}$ such that the walks $e_i ... e_{s_{2n}^{n-1}(i)}$ and 
		$e_{s_{2n}^{n}(i)} ...e_{s_{2n}^{2n-1}(i)}$ belong to the set 
		$B_{G}(\widetilde{e}_{i})$. Then we have the equality 
		$d'(e_{i}) = d'(e_{s_{2n}^{n}(i)})$ for each $i \in \{1, ..., n \}$.
	}
\end{Cor}
\begin{proof}
	From Theorem \ref{PathThm} we obtain the equality 
	$d'(e_{k}) + ... + d'(e_{s_{2n}^{n-1}(k)}) = 
		d'(e_{s_{2n}^{n}(k)}) + ... + d'(e_{s_{2n}^{2n-1}(k)})$ 
	for each $k \in \{1, ..., n\}$ because the both left- and right-hand sides of it are 
	equal to $d'(\widetilde{e}_{k})$. Consideration of the sum of two equalities 
	for indices $k = i$ and $k = s_{2n}^{n+1}(i)$ immediately leads to that 
	$d'(e_{i})=d'(e_{s_{2n}^{n}(i)})$.  
\end{proof}

\noindent
We will call any cycle $C$ in $G \subset \mathcal{P}^{(1)}$ satisfying the condition of 
Corollary \ref{CycleCor} \textit{an isometric even cycle in} $G$ \textit{with respect to}
$\mathcal{P}$. Since the number of edges in $C$ is even one can think that its edges 
$e_{i}$ and $e_{s_{2n}^{n}(i)}$ \textit{are opposite} to each other. Therefore, we have 
that for any isometric even cycle in $G$ the metric $d'$ as in Corollary \ref{CycleCor} 
takes the same value for each pair of opposite edges of $C$.

Using this observation and following \cite[Sec. 2]{Av} we introduce 
\textit{isometric cycle coloring (ic-coloring) of $G \subset \mathcal{P}^{(1)}$ with 
respect to $\mathcal{P}$} which is an assignment of colors, considered just as natural 
numbers, to edges satisfying some property. More precisely, it is defined by 
the following procedure:
\begin{enumerate}[(i)]
	\item Initially all edges of $G$ are uncolored. Pick any edge and give it color $1$,
		set $k = 1$.
	\item Find an uncolored edge that is opposite to an edge colored $k$ in some even
		isometric cycle of $G$. If there is no such edge go to the step (iii), otherwise 
		color the edge $k$ and repeat the step (ii).
	\item If $G$ is not completely colored, pick any uncolored edge, give it color $k+1$,
		set $k \longleftarrow k+1$ and go to the step (ii).
\end{enumerate}
If two edges $e_{1}$ and $e_{2}$ of $G$ have the same color in the result of the algorithm
then we will write that $e_{1} \sim e_{2}$. Note that the algorithm has some uncertainty 
in the process of choosing uncolored edges. However, the step (ii) implies that if 
some of results of the algorithm assigns the same color to given two edges then it will be 
so for these two edges for any result of the algorithm. In particular, it means that 
the algorithm always induces the same partition of the set of edges into color classes, so 
the relation $\sim$ is actually an equivalence relation on the edges of the graph $G$ 
which depends only on $G$ and $\mathcal{P}$. From this it follows that the number of 
colors used by the algorithm is correctly defined number depending only on $G$ and 
$\mathcal{P}$. We will call a graph $G \subset \mathcal{P}^{(1)}$ 
\textit{$k$-ic-colorable in} $\mathcal{P}$ if exactly $k$ colors are used in 
the algorithm.

\begin{Teo}\label{ExtremalMetrics}
	\textit{If a bypassing subgraph $G \subset \mathcal{P}^{(1)}$ is 1-ic-colorable 
		in $\mathcal{P}$ then the metric $d_{G}$ on $\mathcal{P}$ is extreme.
	}
\end{Teo}
\begin{proof}
	Assume that there is a decomposition $d_{G} = d + d'$ for some metrics $d'$, $d''$ 
	on $\mathcal{P}$. Taking into account the result of Corollary \ref{CycleCor}
	the assumption that $G$ is 1-ic-colorable in $\mathcal{P}$ implies that $d'$ takes 
	the same value for all edges of the graph $G$, i. e. $d'|_{E(G)} \equiv \text{const}$. 
	Note that for any edge $e \in E(G)$ we have $d_{G}(e)=1$, so there exists 
	a multiplier $\lambda \in \mathbb{R}_{>0}$ such that $d'(e) = \lambda \cdot d_{G}(e)$ 
	for any $e \in E(G)$.

	Now consider the induction by the value of the metric $d_{G}$ on edges. 
	More precisely, denote by $E_{k}$, $k \geq 1$, the subset of $\mathcal{P}(1)$ 
	consisted of all edges $e \in \mathcal{P}(1)$ such that $d_{G}(e) \leq k$. 
	We have that $E_{1} = E(G)$, $E_{k} \subset E_{k+1}$ for any $k \geq 1$ and 
	$\mathcal{P}(1) = \bigcup_{k=1}^{\infty} E_{k}$. As it was shown $d'|_{E_{1}}$ is 
	proportional to $d_{G}|_{E_{1}}$ with multiplier $\lambda$, this is the base of 
	induction. Now suppose that $d'|_{E_{k}}$ is proportional to $d_{G}|_{E_{k}}$ with 
	multiplier $\lambda$ for any $k \leq k_{0}$. Take any edge $e$ from 
	$E_{k_{0}+1} \setminus E_{k_{0}}$. By the construction, we have that 
	$d_{G}(e) = k_{0}+1$ which means that there is a walk $p \in B_{G}(e)$ consisted of 
	$k_{0}+1$ edges. By the definition, there exists the sequence of walks 
	$p_{1}$, ..., $p_{k_{0}+1}$ such that $p_{1} = p$, $p_{k_{0}+1} = e$ and 
	$p_{i+1}$ is the elementary deformation of $p_{i}$ along some triangle 
	$\Delta_{i} \in \mathcal{P}(2)$, $1 \leq i \leq k_{0}$. In particular, it means that 
	the walk $p_{k_{0}}$ consists of two edges $\gamma_{1}$, $\gamma_{2}$ and it contracts 
	to $e$ along the triangle $\Delta_{k_{0}}$ whose facets are the edges $\gamma_{1}$,
	$\gamma_{2}$, $e$. Repeating the ideas of the proof of Theorem \ref{PathThm} we obtain
	that $d_{G}(\gamma_{1}) + d_{G}(\gamma_{2}) = d_{G}(e)$ which implies that
	$d_{G}(\gamma_{1})$, $d_{G}(\gamma_{2}) < k_{0} + 1$, so $\gamma_{1}$, 
	$\gamma_{2} \in E_{k_{0}}$. On the other hand, from Lemma \ref{TightConstraints} 
	it follows that 
	$d'(e) = d'(\gamma_{1}) + d'(\gamma_{2}) = 
		\lambda \cdot d_{G}(\gamma_{1}) + \lambda \cdot d_{G}(\gamma_{2}) = 
		\lambda \cdot (d_{G}(\gamma_{1}) + d_{G}(\gamma_{2})) = \lambda \cdot d_{G}(e)$.
	Therefore, we have that $d'$ is proportional to $d_{G}$ with multiplier $\lambda$ over 
	$E_{k_{0}+1}$ as well. By induction we obtain that $d'$ is proportional to $d_{G}$ 
	with multiplier $\lambda$ for all edges of $\mathcal{P}$.
	Therefore, the metric $d_{G}$ is extreme.
\end{proof}

The partial case of this theorem for $\mathcal{P}=\mathcal{K}_{n}$ completely coincides
with \cite[Theorem 2.4]{Av}. In particular, \cite[Theorem 3.2]{Av} gives
the large class of $G \subset \mathcal{K}_{n}^{(1)}$ 1-ic-colorable in $\mathcal{K}_{n}$
which correspond to extreme metrics on $\mathcal{K}_{n}$. In its turn, they correspond to
facets of the Graev polytope $P(A_{n+1}), n \geq 1$. We will also present some class of
bypassing subgraphs $G \subset \overline{\mathcal{K}}_{n}^{(1)}$ which are 1-ic-colorable 
in $\overline{\mathcal{K}}_{n}$ giving extreme metrics on $\overline{\mathcal{K}}_{n}$ 
according to Theorem \ref{ExtremalMetrics}. Firstly, recall that any pair of vertices of 
$\overline{\mathcal{K}}_{n}$, say $v_{i}$ and $v_{j}$, is connected by two edges which
were denoted by $e_{ij}$ and $e^{i}_{j}$, respectively. Therefore, we have graph 
decomposition $\overline{\mathcal{K}}_{n}^{(1)} = K_{n}^{-} \cup K_{n}^{+}$ where 
$K_{n}^{-}$ is the subgraph contained all edges of the form $e_{ij}$, while $K_{n}^{+}$ 
contains the edges of the form $e^{i}_{j}$. By the construction of
$\overline{\mathcal{K}}_{n}$, both $K_{n}^{-}$ and $K_{n}^{+}$ are isomorphic to 
the complete graph $K_{n}$.

\begin{Teo}\label{ExtremalMetricsInDn}
	\textit{Let $n$ be a natural number such that $n \geq 5$ and the number $(n-1)$ is not
		divided by 3. Then any spanning subgraph 
		$G \subset K_{n}^{+} \subset \overline{\mathcal{K}}_{n}^{(1)}$ which is isomorphic 
		to $(CH)^{(1)}$ for some Hamiltonian graph $H$ is bypassing and 1-ic-colorable 
		in $\overline{\mathcal{K}}_{n}$.
	}
\end{Teo}
\begin{proof}
	Without loss of generality we can assume that $V(H) = \{ v_{1}, ..., v_{n-1}\}$ and 
	a Hamiltonian cycle in $H$ which we denote by $C$ is giving by the sequence of edges
	$e_{1}^{2} e_{2}^{3} ... e_{n-2}^{n-1} e_{n-1}^{1}$. In this case we have that 
	$V(G) \setminus V(H) = \{ v_{n} \}$ and 
	$E(G) \setminus E(H) = \{ e_{i}^{n} \ | \ 1 \leq i \leq n-1\}$. 
	Denote the subgraph of $K_{n}^{-}$ induced by the subset of vertices 
	$\{v_{1}, ..., v_{n-1}\}$ by $K_{n-1}^{-} \subset K_{n}^{-}$ and, similarly, define 
	the subgraph $K_{n-1}^{+} \subset K_{n}^{+}$. Firstly, we prove that $G$ is bypassing 
	in $\overline{\mathcal{K}}_{n}$. Note that the edges not belonging to $G$ form the set 
	which can be decomposed to the disjoint union as follows: 
	$E(K_{n-1}^{-}) \sqcup \big( E(K_{n}^{-}) \setminus E(K_{n-1}^{-}) \big) 
		\sqcup \big( E(K_{n-1}^{+}) \setminus E(H) \big)$.

	Consider any edge from $E(K_{n-1}^{-})$ which, by construction, has the form $e_{ij}$ 
	for some $1 \leq i < j \leq n-1$. Such edge always has the bypassing contractable walk
	$e_{i}^{n} e_{j}^{n}$. Both edges of this walk belong to the set 
	$E(G) \setminus E(H)$, so the walk lies in $G$. Next, each edge in the set 
	$E(K_{n}^{-}) \setminus E(K_{n-1}^{-})$ has the form $e_{in}$ where 
	$1 \leq i \leq n-1$. Note that the vertex $v_{i}$ belongs to the cycle 
	$C \subset H \subset G$, so the graph $G$ contains the edge $e_{i}^{s_{n-1}(i)}$ 
	by assumption. On the other hand, the edge $e_{s_{n-1}(i)}^{n}$ belongs to the set 
	$E(G) \setminus E(H)$. Therefore, the walk $e_{i}^{s_{n-1}(i)} e_{s_{n-1}(i)}^{n}$ 
	lies in $G$ and, moreover, it is a bypassing contractable walk of the edge $e_{in}$.

	Finally, consider arbitrary edge from $E(K_{n-1}^{+}) \setminus E(H)$. It has the form
	$e_{i}^{j}$ for some $1 \leq i < j \leq n-1$. Since the vertices $v_{i}$ and $v_{j}$
	belong to the cycle $C$, this cycle also contains the walk $e_{i}^{i+1}...e_{j-1}^{j}$
	connecting $v_{i}$ with $v_{j}$. One can show that any walk in $K_{n}^{+}$ connecting
	$v_{i}$ and $v_{j}$ which contains odd number of edges can be contracted to the edge
	$e_{i}^{j}$ along $\overline{\mathcal{K}}_{n}$. 
	So if the walk $e_{i}^{i+1}...e_{j-1}^{j}$ contains the odd number of edges then it is 
	a contractable bypassing walk of the edge $e_{i}^{j}$ which lies in $G$. Otherwise, 
	if this walk has even number of edges, we can consider the slightly modified walk 
	$e_{i}^{n} e_{n}^{i+1} e_{i+1}^{i+2} ... e_{j-1}^{j}$ which already contains 
	odd number of edges and still lies in $G$. Therefore, we obtain that $G$ is 
	a bypassing subgraph in $\mathcal{P}$.

	Now we need to prove that $G$ is 1-ic-colorable. To do this we consider for any vertex
	$v_{i} \in V(H) = V(C)$ the 4-cycle 
	$C_{i} = e_{i}^{s_{n-1}(i)} e_{s_{n-1}(i)}^{s_{n-1}^{2}(i)} 
		e_{s_{n-1}^{2}(i)}^{n} e_{n}^{i}$ 
	with the vertex sequence $v_{i} v_{s_{n-1}(i)} v_{s_{n-1}^{2}(i)} v_{n}$. One can show 
	that these cycles are isometric in $G \subset \overline{\mathcal{K}}_{n}$. Indeed, 
	the pair of vertices $v_{i}$ and $v_{s_{n-1}^{2}(i)}$ is connected by the edge 
	$e_{i s_{n-1}^{2}(i)} \not\in E(G)$ for which we have that the walks 
	$e_{i}^{s_{n-1}(i)} e_{s_{n-1}(i)}^{s_{n-1}^{2}(i)}$ and 
	$e_{s_{n-1}^{2}(i)}^{n} e_{n}^{i}$ belong to the set $B_{G}(e_{i s_{n-1}^{2}(i)})$ 
	since $\overline{\mathcal{K}}_{n}$ contains triangles 
	$\Delta_{i s_{n-1}^{2}(i)}^{s_{n-1}(i)}$ and $\Delta_{i s_{n-1}^{2}(i)}^{n}$. 
	The similar argument holds for the pair of vertices $v_{s_{n-1}(i)}$ and $v_{n}$, 
	so $C_{i}$ is even isometric cycle in $G \subset \overline{\mathcal{K}}_{n}$ for each 
	$1 \leq i \leq n-1$. According to the ic-coloring procedure it means that the opposite 
	edges of the cycle $C_{i}$ have the same colors, i. e. 
	$e_{i}^{s_{n-1}(i)} \sim e_{s_{n-1}^{2}(i)}^{n}$ and 
	$e_{s_{n-1}(i)}^{s_{n-1}^{2}(i)} \sim e_{n}^{i}$ for any $1 \leq i \leq n-1$. 
	It implies that 
	$e_{i}^{s_{n-1}(i)} \sim e_{s_{n-1}^{2}(i)}^{n} \sim 
		e_{s_{n-1}^{3}(i)}^{s_{n-1}^{4}(i)}$, so we have that 
	$e_{i}^{s_{n-1}(i)} \sim e_{s_{n-1}^{3}(i)}^{s_{n-1}^{4}(i)}$ for any 
	$1 \leq i \leq n-1$. By the assumption the number $n-1$ is not divided by $3$, 
	so iterations of the operator $s_{n-1}^{3}$ applied to any $i \in \{ 1, ..., n-1 \}$ 
	cover all the set $\{ 1, ..., n-1 \}$. It means that we have that 
	$e_{i}^{s_{n-1}(i)} \sim e_{j}^{s_{n-1}(j)}$ for any $1 \leq i < j \leq n-1$, i. e.
	all edges of the Hamiltonian cycle $C$ have the same color. Due to that
	$e_{i}^{s_{n-1}(i)} \sim e_{s_{n-1}^{2}(i)}^{n}$ for any $1 \leq i \leq n-1$ we also
	obtain that the edges of $C$ and edges from the set $E(G) \setminus E(H)$ belong to 
	the same color class. Finally, consider the edges from $E(H) \setminus E(C)$. 
	Each of them has the form $e_{i}^{j}$ where $1 \leq i < j \leq n-1$ and 
	$j \neq s_{n-1}(i)$. In the similar way as above one can show that 
	the 4-cycle $e_{i}^{j} e_{j}^{s_{n-1}(j)} e_{s_{n-1}(j)}^{n} e_{n}^{i}$ is isometric 
	in $G \subset \overline{\mathcal{K}}_{n}^{(1)}$ which means that 
	$e_{i}^{j} \sim e_{s_{n-1}(j)}^{n}$. Therefore, the edges from the sets 
	$E(C)$, $E(G) \setminus E(H)$ and $E(H) \setminus E(C)$ are all in 
	the same color class, so $G$ is 1-ic-colorable.
\end{proof}

%------
% Insert acknowledgments and information
% regarding funding at the end of the last
% section, i.e., right before the bibliography.
%------

\begin{ack}
I would like to thank D. V. Alekseevsky for useful discussions on the different topics
related to the paper.
\end{ack}

\begin{funding}
The work was supported by the Theoretical Physics and Mathematics Advancement Foundation 
"BASIS".
\end{funding}

%------
% Insert the bibliography.
%------

\end{document}